\documentclass{amsart}

\usepackage[utf8]{inputenc}

\usepackage[normalem]{ulem}
\usepackage{mathtools}
\usepackage{mathrsfs}  
\usepackage[arrow,curve,matrix]{xy}

\usepackage{scalerel}

\usepackage{comment}

\newif\ifdraft
\draftfalse

\usepackage{amsmath,amsfonts,amsthm,mathrsfs}
\usepackage{amssymb}
\usepackage{mathrsfs}

\usepackage[bookmarks]{hyperref}
\usepackage[usenames,dvipsnames]{xcolor}
\hypersetup{colorlinks=true,citecolor=NavyBlue,linkcolor=Maroon,urlcolor=Orange}

\usepackage[alphabetic,initials]{amsrefs}

\usepackage{enumitem}

\usepackage{chngcntr}

\ifdraft
\usepackage[notcite,notref,color]{showkeys}

\definecolor{labelkey}{gray}{0.5}
\fi

\usepackage{tikz}

\usepackage{tikz-cd}
\tikzset{commutative diagrams/arrow style=math font}

\usetikzlibrary{matrix,arrows}
\newlength{\myarrowsize} 

\pgfarrowsdeclare{cmto}{cmto}{
	\pgfsetdash{}{0pt} 
	\pgfsetbeveljoin 
	\pgfsetroundcap 
	\setlength{\myarrowsize}{0.6pt}
	\addtolength{\myarrowsize}{.5\pgflinewidth}
	\pgfarrowsleftextend{-4\myarrowsize-.5\pgflinewidth} 
	\pgfarrowsrightextend{.8\pgflinewidth}
}{
	\setlength{\myarrowsize}{0.6pt} 
  	\addtolength{\myarrowsize}{.5\pgflinewidth}  
	\pgfsetlinewidth{0.5\pgflinewidth}
	\pgfsetroundjoin
	\pgfpathmoveto{\pgfpoint{1.5\pgflinewidth}{0}}
	\pgfpatharc{-109}{-170}{4\myarrowsize}
	\pgfpatharc{10}{189}{0.58\pgflinewidth and 0.2\pgflinewidth}
	\pgfpatharc{-170}{-115}{4\myarrowsize+\pgflinewidth}
	\pgfpathclose
	\pgfusepathqfillstroke
	\pgfpathmoveto{\pgfpoint{1.5\pgflinewidth}{0}}
	\pgfpatharc{109}{170}{4\myarrowsize}
	\pgfpatharc{-10}{-189}{0.58\pgflinewidth and 0.2\pgflinewidth}
	\pgfpatharc{170}{115}{4\myarrowsize+\pgflinewidth}
	\pgfpathclose
	\pgfusepathqfillstroke
	\pgfsetlinewidth{2\pgflinewidth}
}

\pgfarrowsdeclare{cmonto}{cmonto}{
	\pgfsetdash{}{0pt} 
	\pgfsetbeveljoin 
	\pgfsetroundcap 
	\setlength{\myarrowsize}{0.6pt}
	\addtolength{\myarrowsize}{.5\pgflinewidth}
	\pgfarrowsleftextend{-4\myarrowsize-.5\pgflinewidth} 
	\pgfarrowsrightextend{.8\pgflinewidth}
}{
	\setlength{\myarrowsize}{0.6pt} 
  	\addtolength{\myarrowsize}{.5\pgflinewidth}  
	\pgfsetlinewidth{0.5\pgflinewidth}
	\pgfsetroundjoin
	\pgfpathmoveto{\pgfpoint{1.5\pgflinewidth}{0}}
	\pgfpatharc{-109}{-170}{4\myarrowsize}
	\pgfpatharc{10}{189}{0.58\pgflinewidth and 0.2\pgflinewidth}
	\pgfpatharc{-170}{-115}{4\myarrowsize+\pgflinewidth}
	\pgfpathclose
	\pgfusepathqfillstroke
	\pgfpathmoveto{\pgfpoint{1.5\pgflinewidth}{0}}
	\pgfpatharc{109}{170}{4\myarrowsize}
	\pgfpatharc{-10}{-189}{0.58\pgflinewidth and 0.2\pgflinewidth}
	\pgfpatharc{170}{115}{4\myarrowsize+\pgflinewidth}
	\pgfpathclose
	\pgfusepathqfillstroke
	\pgfpathmoveto{\pgfpoint{1.5\pgflinewidth-0.3em}{0}}
	\pgfpatharc{-109}{-170}{4\myarrowsize}
	\pgfpatharc{10}{189}{0.58\pgflinewidth and 0.2\pgflinewidth}
	\pgfpatharc{-170}{-115}{4\myarrowsize+\pgflinewidth}
	\pgfpathclose
	\pgfusepathqfillstroke
	\pgfpathmoveto{\pgfpoint{1.5\pgflinewidth-0.3em}{0}}
	\pgfpatharc{109}{170}{4\myarrowsize}
	\pgfpatharc{-10}{-189}{0.58\pgflinewidth and 0.2\pgflinewidth}
	\pgfpatharc{170}{115}{4\myarrowsize+\pgflinewidth}
	\pgfpathclose
	\pgfusepathqfillstroke
	\pgfsetlinewidth{2\pgflinewidth}
}

\pgfarrowsdeclare{cmhook}{cmhook}{
	\pgfsetdash{}{0pt} 
	\pgfsetbeveljoin 
	\pgfsetroundcap 
	\setlength{\myarrowsize}{0.6pt}
	\addtolength{\myarrowsize}{.5\pgflinewidth}
	\pgfarrowsleftextend{-4\myarrowsize-.5\pgflinewidth} 
	\pgfarrowsrightextend{.8\pgflinewidth}
}{
	\setlength{\myarrowsize}{0.6pt} 
  	\addtolength{\myarrowsize}{.5\pgflinewidth}  
 	\pgfsetdash{}{0pt}
	\pgfsetroundcap
	\pgfpathmoveto{\pgfqpoint{0pt}{-4.667\pgflinewidth}}
	\pgfpathcurveto
    {\pgfqpoint{4\pgflinewidth}{-4.667\pgflinewidth}}
    {\pgfqpoint{4\pgflinewidth}{0pt}}
    {\pgfpointorigin}
	\pgfusepathqstroke
}

\newenvironment{diagram*}[2]{%
\[%
\begin{tikzpicture}[>=cmto,baseline=(current bounding box.center),%
	to/.style={->,font=\scriptsize,cap=round},%
	into/.style={cmhook->,font=\scriptsize,cap=round},%
	onto/.style={-cmonto,font=\scriptsize,cap=round},%
	math/.style={matrix of math nodes, row sep=#2, column sep=#1,%
		text height=1.5ex, text depth=0.25ex}]%
}{%
\end{tikzpicture}%
\]%
\ignorespacesafterend%
}

\DeclareMathOperator{\Exc}{Exc}

\DeclareMathOperator{\Hilb}{Hilb}

\DeclareMathOperator{\Spec}{Spec}

\DeclareMathOperator{\id}{id}

\renewcommand{\Im}{\operatorname{Im}}

\DeclareMathOperator{\Supp}{Supp}
\DeclareMathOperator{\codim}{codim}

\DeclareMathOperator{\Sym}{Sym}

\DeclareMathOperator{\PGL}{PGL}
\DeclareMathOperator{\SL}{SL}

\DeclareMathOperator{\Var}{Var}

\newcommand{\sA}{\scr{A}}
\newcommand{\sB}{\scr{B}}

\newcommand{\sE}{\scr{E}}
\newcommand{\sL}{\scr{L}}
\newcommand{\sF}{\scr{F}}
\newcommand{\sG}{\scr{G}}
\newcommand{\sH}{\scr{H}}

\newcommand{\sM}{\scr{M}}
\newcommand{\sN}{\scr{N}}
\newcommand{\sO}{\scr{O}}

\newcommand{\sQ}{\scr{Q}}

\newcommand{\sW}{\scr{W}}

\newcommand{\theoremref}[1]{\hyperref[#1]{Theorem~\ref*{#1}}}
\newcommand{\lemmaref}[1]{\hyperref[#1]{Lemma~\ref*{#1}}}
\newcommand{\definitionref}[1]{\hyperref[#1]{Definition~\ref*{#1}}}
\newcommand{\propositionref}[1]{\hyperref[#1]{Proposition~\ref*{#1}}}
\newcommand{\conjectureref}[1]{\hyperref[#1]{Conjecture~\ref*{#1}}}
\newcommand{\corollaryref}[1]{\hyperref[#1]{Corollary~\ref*{#1}}}
\newcommand{\exampleref}[1]{\hyperref[#1]{Example~\ref*{#1}}}

\makeatletter
\let\old@caption\caption
\renewcommand*{\caption}[1]{%
	\setcounter{figure}{\value{equation}}%
	\stepcounter{equation}%
	\old@caption{#1}\relax%
}
\makeatother

\newcounter{intro}

\newtheorem{intro-conjecture}[intro]{Conjecture}
\newtheorem{intro-corollary}[intro]{Corollary}
\newtheorem{intro-theorem}[intro]{Theorem}

\newcommand{\parref}[1]{\hyperref[#1]{\S\ref*{#1}}}

\makeatletter
\newcommand*\if@single[3]{%
  \setbox0\hbox{${\mathaccent"0362{#1}}^H$}%
  \setbox2\hbox{${\mathaccent"0362{\kern0pt#1}}^H$}%
  \ifdim\ht0=\ht2 #3\else #2\fi
  }
\newcommand*\rel@kern[1]{\kern#1\dimexpr\macc@kerna}
\newcommand*\widebar[1]{\@ifnextchar^{{\wide@bar{#1}{0}}}{\wide@bar{#1}{1}}}
\newcommand*\wide@bar[2]{\if@single{#1}{\wide@bar@{#1}{#2}{1}}{\wide@bar@{#1}{#2}{2}}}
\newcommand*\wide@bar@[3]{%
  \begingroup
  \def\mathaccent##1##2{%
    \if#32 \let\macc@nucleus\first@char \fi
    \setbox\z@\hbox{$\macc@style{\macc@nucleus}_{}$}%
    \setbox\tw@\hbox{$\macc@style{\macc@nucleus}{}_{}$}%
    \dimen@\wd\tw@
    \advance\dimen@-\wd\z@
    \divide\dimen@ 3
    \@tempdima\wd\tw@
    \advance\@tempdima-\scriptspace
    \divide\@tempdima 10
    \advance\dimen@-\@tempdima
    \ifdim\dimen@>\z@ \dimen@0pt\fi
    \rel@kern{0.6}\kern-\dimen@
    \if#31
      \overline{\rel@kern{-0.6}\kern\dimen@\macc@nucleus\rel@kern{0.4}\kern\dimen@}%
      \advance\dimen@0.4\dimexpr\macc@kerna
      \let\final@kern#2%
      \ifdim\dimen@<\z@ \let\final@kern1\fi
      \if\final@kern1 \kern-\dimen@\fi
    \else
      \overline{\rel@kern{-0.6}\kern\dimen@#1}%
    \fi
  }%
  \macc@depth\@ne
  \let\math@bgroup\@empty \let\math@egroup\macc@set@skewchar
  \mathsurround\z@ \frozen@everymath{\mathgroup\macc@group\relax}%
  \macc@set@skewchar\relax
  \let\mathaccentV\macc@nested@a
  \if#31
    \macc@nested@a\relax111{#1}%
  \else
    \def\gobble@till@marker##1\endmarker{}%
    \futurelet\first@char\gobble@till@marker#1\endmarker
    \ifcat\noexpand\first@char A\else
      \def\first@char{}%
    \fi
    \macc@nested@a\relax111{\first@char}%
  \fi
  \endgroup
}
\makeatother

\DeclareFontFamily{OMS}{rsfs}{\skewchar\font'60}
\DeclareFontShape{OMS}{rsfs}{m}{n}{<-5>rsfs5 <5-7>rsfs7 <7->rsfs10 }{}
\DeclareSymbolFont{rsfs}{OMS}{rsfs}{m}{n}
\DeclareSymbolFontAlphabet{\scr}{rsfs}

\numberwithin{figure}{section}

\DeclareMathOperator{\disc}{disc}

\DeclareMathOperator{\Gal}{Gal}

\DeclareMathOperator{\Ob}{Ob}
\DeclareMathOperator{\pr}{pr}
\DeclareMathOperator{\R}{R}

\DeclareMathOperator{\mfM}{\mathfrak{M}}
\DeclareMathOperator{\dR}{\mathbf{R}}

\newcommand{\bC}{\mathbb{C}}

\newcommand{\bN}{\mathbb{N}}

\newcommand{\wtilde}{\widetilde}

\newcommand{\hooklongrightarrow}{\lhook\joinrel\longrightarrow}

\theoremstyle{plain}
\newtheorem{theorem}{Theorem}[section]

\newtheorem{proposition}[theorem]{Proposition}
\newtheorem{corollary}[theorem]{Corollary}
\newtheorem{lemma}[theorem]{Lemma}

\theoremstyle{definition}
\newtheorem{definition}[theorem]{Definition}

\newtheorem{explanation}[theorem]{Explanation}
\newtheorem{assumption}[theorem]{Assumption}
\theoremstyle{remark}
\newtheorem{remark}[theorem]{Remark}
\newtheorem{observation}[theorem]{Observation}
\newtheorem{notation}[theorem]{Notation}

\newtheorem{set-up}[theorem]{Set-up}
\newtheorem{claim}[theorem]{Claim}
\newtheorem{subclaim}[theorem]{Subclaim}

\newtheorem{remark-def}[theorem]{Remark-Definition}
\newtheorem{theorem-def}[theorem]{Theorem-Definition}

\setlist[enumerate]{label=(\thetheorem.\arabic*), before={\setcounter{enumi}{\value{equation}}}, after={\setcounter{equation}{\value{enumi}}}}

\numberwithin{equation}{theorem}

\makeatletter
\let\amsmath@bigm\bigm

\renewcommand{\bigm}[1]{%
  \ifcsname fenced@\string#1\endcsname
    \expandafter\@firstoftwo
  \else
    \expandafter\@secondoftwo
  \fi
  {\expandafter\amsmath@bigm\csname fenced@\string#1\endcsname}%
  {\amsmath@bigm#1}%
}
\newcommand{\DeclareFence}[2]{\@namedef{fenced@\string#1}{#2}}
\makeatother

\DeclareFence{\mid}{|}

\begin{document}

\title[birational geometry of smooth families]
{birational geometry of smooth families of varieties admitting good minimal models}

\author{Behrouz Taji}
\address{Behrouz Taji, School of Mathematics and Statistics - Red Centre,
The University of New South Wales, NSW 2052 Australia}
\email{\href{mailto:b.taji@unsw.edu.au}{b.taji@unsw.edu.au}}
\urladdr{\href{https://web.maths.unsw.edu.au/~btaji/}{https://web.maths.unsw.edu.au/~btaji/}}




\setlength{\parskip}{0.19\baselineskip}


\begin{abstract}
In this paper we study families of projective 
manifolds with good minimal models.
After constructing a suitable moduli functor for polarized varieties with 
canonical singularities, we show that, if not birationally isotrivial, the base spaces of such families 
support subsheaves of log-pluridifferentials
with positive Kodaira dimension. Consequently we prove that, 
over special base schemes, families of this type can only be birationally isotrivial 
and, as a result, confirm a conjecture of Kebekus and Kov\'acs.
\end{abstract}

\maketitle

\tableofcontents

\section{Introduction and main results}
\label{sect:Section1-Introduction}

A conjecture of Shafarevich and Viehweg 
predicts that smooth projective families of manifolds with ample canonical bundle (canonically polarized)
whose algebraic structure maximally varies have base spaces of log-general type.
This conjecture was settled through the 
culmination of works of many people, including Parshin \cite{Parshin68}, Arakelov \cite{Arakelov71}, 
Kov\'acs~\cite{Kovacs02}, Viehweg-Zuo~\cite{VZ02}, 
Kebekus-Kov\'acs~\cite{KK08}, \cite{KK10}, Patakfalvi~\cite{MR2871152}
and Campana-P\u{a}un~\cite{CP14},~\cite{CP16}.

More recently, triggered by~\cite{Vie-Zuo01} and the result of Popa-Schnell~\cite{PS15}, 
it has been speculated that far more general results should hold
for a considerably larger category of projective manifolds; those with \emph{good minimal models}.
In other words there is a conjectural connection between (birational) variation 
in smooth projective family of non-uniruled manifolds and global geometric properties 
of their base. 
In this setting the most general conjecture is a generalization of a conjecture of Campana which we resolve in 
this paper.

\smallskip

\begin{theorem}[isotriviality over special base]
\label{thm:iso}
Let $U$ and $V$ be smooth quasi-projective varieties. If $V$ is special, then every smooth projective family 
$f_U: U\to V$ of varieties with good minimal models is birationally isotrivial.
\end{theorem}

\smallskip

In this article all schemes are over $\bC$. 
Following ~\cite{Viehweg83} and Kawamata~\cite[pp.~5--6]{Kawamata85} 
we define $\Var(f_U)$ by the transcendence degree of a \emph{minimal closed 
field of definition} $K$ for $f_U$. We note that $K$ is the minimal (in terms of inclusion)
algebraically closed field in the algebraic closure $\overline{\bC(V)}$ of the function field 
$\bC(V)$ for which there is a $K$-variety $W$ such that 
$U\times_V \Spec\big( \overline{\bC(V)} \big)$ is birationally equivalent 
to $W\times_{Spec(K)} \Spec(\overline{\bC(V)})$ (see Definition~\ref{def:VAR}). 
By \emph{birationally isotrivial} we mean $\Var(f_U)=0$. 
We recall that an $n$-dimensional smooth quasi-projective variety $V$ is called \emph{special} if, for $1\leq p\leq n$, every 
invertible subsheaf $\sL \subseteq  \Omega^p_{B} (\log D)$ verifies the inequality 
$\kappa(\sL) < p$, where $(B,D)$ is any smooth compatification of $V$, cf.~\cite{Cam04}.
Varieties with zero Kodaira dimension~\cite[Thm~5.1]{Cam04} and rationally connected manifolds
are important examples of special varieties.

In \cite[Sect.~5]{Taj18}
it was shown that, thanks to Campana's results on the orbifold $C_{n.m}$ conjecture, 
once Theorem~\ref{thm:iso} is established the following conjecture of Kebekus-Kov\'acs~\cite[Conj.~1.6]{KK08}
(formulated in this general form in~\cite{PS15}) follows as a consequence.

\begin{theorem}[resolution of Kebekus-Kov\'acs Conjecture]
\label{thm:kk}
Let $f_U: U\to V$ be a smooth projective family of manifolds with good minimal models. Then, either
\begin{enumerate}
\item $\kappa(V) = -\infty$ and $\Var(f_U) < \dim V$, or 
\item $\kappa(V)\geq 0$ and $\Var(f_U)\leq \kappa(V)$.
\end{enumerate}
\end{theorem}

When $\Var(f_U)$ is maximal ($\Var(f_U) = \dim V$), these conjectures
are all equivalent to Viehweg's original conjecture
generalized to the setting of manifolds admitting good minimal models. The latter is a 
result of~\cite{PS15} combined with~\cite{CP14}. 

For canonically polarized fibers Theorem~\ref{thm:iso}
was settled in~\cite{Taji16}. A key component 
of the proof was the following celebrated result of~\cite{Vie-Zuo03a} for the base space
of a projective family $f_U: U \to V$ of canonically polarized manifolds.

\begin{quote}
$(*)$  There are $k\in \bN$ and an invertible subsheaf $\sL \subseteq \big(  \Omega^1_B(\log D) \big)^{\otimes k}$ 
     such that $\kappa(\sL) \geq \Var(f_U)$.
\end{quote}

Establishing $(*)$ in the more general context of projective manifolds with good minimal
models has been an important goal in this topic. In its absence, a weaker result was established 
in~\cite{Taj18} where it was shown that for projective families with good minimal models we have:

\begin{quote}
$(**)$  There are $k\in \bN$, a pseudo-effective line bundle $\sB$ and a line bundle $\sL$ on $B$, with 
 $(\sL \otimes \sB) \subseteq \big( \Omega^1_B(\log D)  \big)^{\otimes k}$, such that $\kappa(\sL) \geq \Var(f_U)$. 
\end{quote}

Clearly $(**)$ is equivalent to $(*)$ when variation is maximal, in which
case the result is due to~\cite{PS15}, and~\cite{Vie-Zuo01} when the base is
of dimension one. 
But, as it is shown in~\cite{Taj18} and~\cite[Subsect.~4.3]{PS15}, the discrepancy 
between $(*)$ and $(**)$ poses a major obstacle in proving Kebekus-Kov\'acs Conjecture in its
full generality. In this paper we close this gap and prove the following result.

\begin{theorem}\label{thm:sheaves}
Let $f_U: U\to V$ be a smooth, projective and non-birationally isotrivial morphism 
of smooth quasi-projective varieties $U$ and $V$ with positive relative dimension. 
Let $(B,D)$ be a smooth compactification of $V$.
If the fibers of $f_U$ have good minimal models, then 
there exist $k \in \bN$ and an invertible subsheaf $\sL \subseteq \big( \Omega^1_B(\log D)  \big)^{\otimes k}$  
such that $\kappa(B, \sL) \geq \Var(f_U)$.
\end{theorem}

The fundamental reason underlying the difference between the two results $(**)$ and 
Theorem~\ref{thm:sheaves}
is that while the proof of the former makes no use of a suitable moduli space associated to 
a relative minimal model program for the family $f_U$, the improvement in the 
latter heavily depends on a well-behaved moduli functor
that we construct in Section~\ref{sect:Section3-Functor} for any projective family of manifolds
with good minimal models.

\begin{theorem}\label{thm:functor}
Let $f_U: U\to V$ be a smooth projective family of varieties admitting good minimal models.
For every family $f_{U'}:U' \to V$ resulting from a relative good minimal model program for $f_U$, 
after removing a closed subscheme of $V$, there is
an ample line bundle $\sL$ on $U'$ and a moduli functor $\mfM^{[N]}$ such
that 
\[
( f_{U'} : U' \longrightarrow  V , \sL ) \in \mfM^{[N]}_h(V) ,
\]
where $h$ is a fixed Hilbert polynomial. 
Moreover, the functor $\mfM_h^{[N]}$ has a 
coarse moduli space $M_h^{[N]}$
and that $\Var(f_{U})$ is equal to the dimension of the image of $V$ 
under the associated moduli map.
\end{theorem}

The existence of a functor $\mfM^{[N]}$ as in Theorem~\ref{thm:functor}, 
approximating enough properties of the well-known functor for canonically 
polarized manifolds \cite{Viehweg95} for a prescribed family $f_U$ 
was not known before (see Proposition~\ref{prop:functor} and Theorem~\ref{VarModuli} for more details), 
which explains the focus of \cite[Subsect.~4.3]{PS15} and subsequently \cite{Taj18} 
on the application of Abundance type results to tackle Kebekus-Kov\'acs Conjecture.
The key advantage that Theorem~\ref{thm:functor} offers is that instead of 
constructing $\sL$ at the base of $f: X \to B$ we do so at the level of a
\emph{moduli stack}; a smooth projective variety $Z$ equipped with a generically 
finite morphism to $M_h^{[N]}$ and parametrizing a
new family, now with maximal (birational) variation. 
But once variation is maximal, again $(*)$ and $(**)$ are equivalent and
the pseudo-effective line bundle $\sB$ in $(**)$ can essentially be ignored.

Since there are no maps from $B$ to $Z$, the next difficulty is then to lift 
this big line bundle on $Z$ to a line bundle on $B$. We resolve this problem 
by showing that the construction of such invertible sheaves are in a sense
functorial. More precisely we show that the Hodge theoretic constructions 
in~\cite{Taj18}, from which these line bundle arise, verify various functorial properties
that are sufficiently robust for the construction of the line bundle $\sL$ in Theorem~\ref{thm:sheaves},
using the one constructed at the level of moduli stacks. 
This forms the main content of Section~\ref{sect:Section2-Subsheaves}.

\subsection{Notes on previously known results} 
When dimension of the base and fibers are equal to one, Viehweg's 
hyperbolicity conjecture was proved by Parshin
\cite{Parshin68}, in the compact case,
and in general by Arakelov \cite{Arakelov71}. For higher dimensional fibers and assuming 
that $\dim(V)=1$, this conjecture was confirmed by Kov\'acs~\cite{Kovacs02}, in the canonically polarized case, 
and by Viehweg and Zuo~\cite{Vie-Zuo01} in general.
Over Abelian varieties Viehweg's conjecture was solved by Kov\'acs~\cite{Kovacs97c}.
When $\dim(V)=2$ or $3$, it was resolved by Kebekus and Kov\'acs, cf.~\cite{KK08} and \cite{KK10}.
In the compact case it was settled by Patakfalvi~\cite{MR2871152}.
In the canonically polarized case, and when $\dim V\leq 3$, Theorem~\ref{thm:iso}
is due to Jabbusch and Kebekus~\cite{MR2860268}.
Using $(**)$ Kebekus-Kov\'acs conjecture is settled in~\cite{Taj18} under the 
assumption that $\dim V \leq 5$. More recently Theorem~\ref{thm:iso} for fibers of general type
has appeared in the work of Wei-Wu~\cite{WW20}.


\subsection{Acknowledgements} 
I am very thankful to Fr\'ed\'eric Campana. 
This project was triggered by his visit to Sydney Mathematical Research Institute (SMRI)
and his help was decisive throughout 
the process of writing this paper.
I also owe a special debt of gratitude to S\'andor 
Kov\'acs for answering my multiple questions and providing detailed feedback regarding 
Section~\ref{sect:Section3-Functor} of this paper. 
I would also like to thank Yajnaseni Dutta for helpful comments and Sung Gi Park for  
pointing out a gap in the arguments in an earlier version of Section 2.

\section{Functorial properties of subsheaves of extended variation of Hodge structures}
\label{sect:Section2-Subsheaves}

Our aim in this section is to show that the Hodge theoretic constructions in~\cite{Taj18}
enjoy various functorial properties. These will play a crucial role in the proof of Theorems~\ref{thm:sheaves}
and~\ref{thm:iso} in Section~\ref{sect:Section4-MainThm}.

\begin{notation}[discriminant]\label{notation1}
For a morphism $f:X \to Y$ of quasi-projective varieties with connected fibers, by 
$D_f$ we denote the divisorial part of the discriminant locus $\disc(f)$.
We define $\Delta_f$ to be the maximal reduced divisor supported over $f^{-1}D_f$.

\end{notation}

\subsection{Geometric setup}\label{GeomSetUp}
Let $f:X\to Y$ be a morphism of smooth quasi-projective varieties with connected fibers and
relative dimension $n$. Let $\sM$ be a line bundle on $X$. We will sometimes need the 
extra assumption that 

\begin{equation}\label{eq:extra}
H^0 (T, \mu^*\sM)\neq 0 ,
\end{equation}
for some proper surjective morphism $\mu: T\to X$ from 
a smooth quasi-projective variety $T$.
For example, the assumption (\ref{eq:extra}) is valid when $H^0 (X, \sM^m)\neq 0$, in which case
$T$ can be taken to be any desingularization of the cyclic cover associated 
to a prescribed global section of $\sM^m$~\cite{BG71} (see also~\cite[Prop.~4.1.6]{Laz04-I}).

Now,  let $g: Y^+ \to Y$ be a morphism 
of smooth quasi-projective varieties and set $g': X^+\to X$ to be a strong 
desingularization\footnote{A desingularization that restricts to an isomorphism over the regular locus.}
of $Y^+ \times_Y X$ with the resulting family $f^+:X^+ \to Y^+$. 
Next, we define $\sM^+: = (g')^*\sM$. 
Assuming that (\ref{eq:extra}) holds, let $T^+$ be any smooth quasi-projective variety 
with a birational surjective morphism to a strong desingularization of $(T \times_X X^+)$ with induced 
maps $g'': T^+ \to T$ and 
$\mu^+ : T^+ \to X^+$. By construction we have

$$
H^0( T^+, (\mu^+)^*(\sM^+) )\neq 0.
$$
Finally, we define the two compositions 
$$
h : = f \circ \mu   \; \; \; \; \; \; \; \; \text{and} \; \;\; \;\; \;\; \;\;  h^+ : = f^+ \circ \mu^+.
$$

$$
\xymatrix{
T^+ \ar@/_5mm/[ddr]_{h^+}  \ar[rrrr]^{g''}  \ar[dr]_{\mu^+} &&&&  T \ar[dl]^{\mu} \ar@/^5mm/[ddl]^h\\
& X^+ \ar[d]_{f^+} \ar[rr]^{g'}  &&  X \ar[d]^f \\
& Y^+ \ar[rr]^g  &&   Y
}
$$
We will assume that $\Delta_f$, $\Delta_h$, $\Delta_{f^+}$ and $\Delta_{h^+}$ have simple normal crossing support
(see Notation~\ref{notation1}).

\medskip

\subsection{Hodge theoretic setup}\label{HodgeSetUp}
In the setting of~\ref{GeomSetUp}, after removing subsets of $\codim\geq 2$ from the base, we may assume 
that $D_f$ and $D_h$ also have simple 
normal crossing support.

Let $(\sE= \bigoplus \sE_{i}, \theta)$ be a logarithmic system of Hodge bundles 
underlying the Deligne canonical extension of $\R^n h_* {\underline{\bC}\;}|_{T\setminus \Delta_h}$
(with the fixed interval $[0,1)$). For every $0 \leq p \leq n$, and after removing a closed subset of $Y$ along $D_h$ of 
$\codim_{Y}\geq 2$, let 

\begin{equation*}
\big(  \Omega^{\bullet}_T  (\log \Delta_h)  , F_{T, \bullet}  \big)
 \end{equation*}
be the filtered logarithmic de Rham complex with the decreasing locally free filtration $F_{T, \bullet}$,
with locally free gradings, induced by the 
exact sequence 
\begin{equation}\label{ExactSeq}
0 \longrightarrow  h^* \Omega^1_Y(\log D_h)  \longrightarrow \Omega^1_T (\log \Delta_h) 
\longrightarrow \Omega^1_{T/Y}(\log \Delta_h)
\longrightarrow 0.
\end{equation}

Let $C^p_T$ denote the complex corresponding to $\Omega^p_T(\log \Delta_h)$ defined by the short exact sequence 
$$
0 \longrightarrow  h^* \Omega^1_Y(\log D_h) \otimes \Omega^{p-1}_{T/Y}(\log \Delta_h)  \longrightarrow 
  \frac{\Omega^p_T (\log \Delta_h)}{F^2_{T,p}} \longrightarrow \Omega^p_{T/Y}(\log \Delta_h)
\longrightarrow 0,
$$
given by quotienting out the short exact sequence $0 \to F^1_{T,p} \to F^0_{T,p} \to \Omega^p_{T/Y} (\log \Delta_h)\to 0$
by $F^2_{T,p}$. Thanks to Steenbrink~\cite{Ste76} and Katz-Oda~\cite{KO68} we know there is an 
isomorphism of systems of Hodge bundles
$$
(\sE, \theta) \cong \bigoplus \R^i h_* \Omega^{n-i}_{T/Y}  ( \log \Delta_h),
$$
with the Higgs field of the system on the right defined by the long exact cohomology sequence 
associated to $\dR h_* C^p_T$ 
(which is a distinguished triangle in the bounded derived category of coherent sheaves).

\begin{definition}
Let $\sW$ be an $\sO_Y$-module on a regular scheme $Y$. Then, a $\sW$-valued system 
$(\sF, \tau)$ consists of an $\sO_Y$-module $\sF$ and a sheaf homomorphism $\tau: \sF\to \sW \otimes \sF$ 
that is Griffiths-transversal with respect to an $\sO_Y$-module splitting $\sF= \bigoplus \sF_i$, i.e. 
$\tau: \sF_i \to \sW \otimes \sF_{i+1}$. 
\end{definition}
In particular, when $\sW=\Omega^1_Y$ and $\tau$ is integrable, 
$(\sF, \tau)$ is the usual system of Hodge sheaves.

Following the general strategy of~\cite{Vie-Zuo03a} as we have seen in~\cite{Taj18} we can construct 
a $\Omega^1_Y(\log D_f)$-valued system $(\sF, \tau)$. 
Furthermore, if the assumption 
(\ref{eq:extra}) holds, then there is a map of systems 
$$
\Phi: (\sF, \tau) \longrightarrow  (\sE, \theta).
$$
For the reader's convenience we briefly recall the construction of~\cite[Subsect.~2.2]{Taj18}.
First note that, similar to the construction of $C^p_T$, we can construct $C^p_X$ and consider 
the twisted short exact sequence $C^p_X\otimes \sM^{-1}$.

\begin{proposition}\label{FilteredMap}
In the above setting, over the flat locus of $f$ and $h$, for every $0\leq p \leq \dim(X/Y)$, there is a filtered morphism 

\begin{equation}\label{eq:FILT1}
\mu^* \big(  (\Omega^p_X (\log \Delta_f)  , F^j_{X,p}  ) \otimes \sM^{-1} \big)   \longrightarrow  
  \big(  \Omega^p_T(\log \Delta_h)  , F^j_{T,p}  \big).
\end{equation}
Consequently, there is a morphism of short exact sequences $\mu^*(C^p_X \otimes \sM^{-1}) \to C^p_T$.
\end{proposition}

\begin{proof}
First, we note that the pullback of short exact sequence of locally free sheaves 

$$
0 \longrightarrow    f^* \Omega^1_Y(\log D_f) \longrightarrow \Omega^1_X(\log \Delta_f) \longrightarrow
   \Omega^1_{X/Y} (\log \Delta_f)  \longrightarrow 0 
$$
via $\mu$ is a subsequence of (\ref{ExactSeq}).
Therefore, by the construction of the two filtrations $F_{X,p}$ and $F_{T,p}$, cf.~\cite[Ex.~5.16(c)]{Ha77}, we have 
a filtered morphism 

\begin{equation}\label{eq:FILT2}
\mu^* \big( \Omega_X^p (\log \Delta_f)   , F^j_{X,p}  \big)    \longrightarrow   ( \Omega^p_T (\log \Delta_h) , F^j_{T,p} ) .
\end{equation}
in particular, for $j=1,2$, we have 

\begin{equation}\label{eq:22}
\mu^* F^j_{X,p}  \longrightarrow   F^j_{T,p}
\end{equation}
 with the following commutative diagram 
 
 \begin{equation}\label{eq:11}
 \xymatrix{
0   \ar[r]  &  \mu^*F^1_{X,p}  \ar[r] \ar[d] &  \mu^*F^0_{X,p}
     \ar[r] \ar[d]  &  \mu^*\Omega^p_{X/Y} (\log \Delta_f)  \ar[r] \ar[d]  &    0   \\  
0   \ar[r]    &  F^1_{T,p}  \ar[r]   &  F^0_{T,p}  \ar[r]   &  \Omega^p_{T/Y} (\log \Delta_h) \ar[r]   &   0  . 
}
 \end{equation}
 Now, consider the commutative diagram 
\begin{equation}\label{eq:33}
\xymatrix{
0   \ar[r]  &  F^2_{X,P}  \ar[r] \ar[d] &  F^2_{X/P}  \ar[r] \ar[d]  &  0  \ar[r] \ar[d]  &    0   \\  
0   \ar[r]    &  F^1_{X,p}  \ar[r]   &  F^0_{X,p}  \ar[r]   &  \Omega^p_{X/Y} (\log \Delta_f) \ar[r]   &   0  . 
}
\end{equation}
By the nine lemma,  the diagram (\ref{eq:33}) induces 
$$
0 \longrightarrow   F^1_{X,p} / F^2_{X,p}  \longrightarrow F^0_{X,p} / F^2_{X,p}  
 \longrightarrow \Omega^p_{X/Y} (\log \Delta_f) \longrightarrow 0, 
$$
which we have denoted by $C^p_X$.  By combining (\ref{eq:11}) and (\ref{eq:22}) and the 
functoriality of the nine lemma (in the abelian category of coherent sheaves), 
after pulling back (\ref{eq:33}) by $\mu$ we find 
the morphism 
\begin{equation}\label{eq:STAR}
\mu^* C^p_X  \longrightarrow C^p_T.
\end{equation}

Furthermore, by the assumption (\ref{eq:extra}) we have the 
natural injection $\mu^*\sM^{-1} \hooklongrightarrow \sO_T$.
This implies that there is a filtered injection 
$$
\mu^* \big(  (  \Omega^p_{X/Y} (\log \Delta_f)  , F_{X,p} )   \otimes \sM^{-1}   \big)   
     \longrightarrow  \mu^*( \Omega^p_{X/Y} (\log \Delta_f), F_{X,p} )
$$
which, together with (\ref{eq:FILT2}), establishes (\ref{eq:FILT1}). 
Moreover, after twisting (\ref{eq:33}) by $\sM^{-1}$, again by the nine lemma (and its functoriality) 
we have 
$$
\mu^* \big(C^p_X \otimes \sM^{-1}\big)   \longrightarrow \mu^* C^p_X  .
$$
The proposition now follows from the composition of this latter morphism with (\ref{eq:STAR}).
\end{proof}

Now, let $(\sF,\tau)$
be the system defined by 
$$
\sF_i : = \R^i f_*  \big( \Omega^{n-i}_{X/Y} (\log \Delta_f) \otimes \sM^{-1}  \big),
$$
with each $\tau|_{\sF_i}$ given by the connecting maps in the cohomology sequence 
associated to $\dR f_* ( C^p_X \otimes \sM^{-1} )$. 
By applying $\dR \mu_*$ to the map $\mu^*\big(C^p_X \otimes \sM^{-1}\big) \to C^p_T$
in Proposition~\ref{FilteredMap} we have 
$$
\dR \mu_* \mu^*(C^p_X\otimes \sM^{-1} )  \longrightarrow  \dR \mu_* C^p_T .
$$
Using the (derived) projection formula and the adjunction map $\sO_X \to \dR \mu_* \sO_T$ we thus get 
$$
C^p_X\otimes \sM^{-1} \longrightarrow \dR\mu_* C^p_T,
$$
and consequently the morphism
\begin{equation}\label{eq:KeyMap}
\dR f_*  ( C^p_X \otimes \sM^{-1} ) \longrightarrow  \dR h_*  C^p_T.
\end{equation} 

\medskip

\begin{proposition}[\protect{cf.~\cite[Subsect.~2.2]{Taj18}}]
\label{prop:TajMap}
The morphism (\ref{eq:KeyMap}) induces the commutative diagram

$$
\xymatrix{
\R^{n-p} f_*  \big( \Omega^p_{X/Y} (\log \Delta_f) \otimes \sM^{-1} \big)   \ar[r]^(0.4){\tau} \ar[d]_{\Phi_{n-p}}  &  
    \Omega^1_Y ( \log D_f ) \otimes  \R^{n-p+1} f_* \big(\Omega^{p-1}_{X/Y} (\log \Delta_f) \otimes \sM^{-1} \big) \ar[d]^{i \otimes \Phi_{n-p+1}}  \\
\sE_{n-p}  \ar[r]^(0.4){\theta}   &  \Omega^1_Y(\log D_h) \otimes \sE_{n-p+1},
}
$$
where $i$ is the natural inclusion map. 
The vertical maps on the left define $\Phi: (\sF, \tau)\to (\sE,\theta)$ by $\Phi = \bigoplus \Phi_i$.
Furthermore, $\Phi_0$ is injective.

\end{proposition}

\smallskip

We can replicate these construction for $f^+:X^+\to Y^+$. 
That is, assuming that $D_{f^+}$ and $D_{h^+}$ have simple normal crossing support and 
after removing a closed subscheme of $Y^+$ along $D_{f^+}$ of $\codim_{Y^+}\geq 2$ (if necessary), 
we can define two systems $(\sF^+, \tau^+)$, $(\sE^+, \theta^+)$ whose graded pieces are given by
$$
\sF^+_i  = \R^i f_*^+  \bigl(  \Omega^{n-i}_{X^+/Y^+} (\log \Delta_{f^+}) \otimes (\sM^+)^{-1}  \bigr) \; \; \; , \; \; \;  
\sE_i^+ =  \R^i h^+_*  \big( \Omega^{n-i}_{T^+/Y^+}  (\log \Delta_{h^+})  \big).
$$
Similarly we can also define a morphism of systems $\Phi^+: (\sF^+, \tau^+) \to (\sE^+, \theta^+)$ on $Y^+$ 

\medskip
\subsection{Functoriality.~I}\label{subsect:functorial}
In the setting of Subsection~\ref{GeomSetUp}, let $X' : = X\times_Y Y^+$
and $\pi: X^+ \to X'$ be the strong resolution defining $g'$ as the composition $\sigma \circ \pi$:
$$
\xymatrix{
X^+ \ar@/^5mm/[rrr]^{g'} \ar[r]^{\pi}  \ar[dr]_{f^+} &  X'  \ar[rr]^{\sigma} \ar[d]^{f'} &&  X \ar[d]^{f} \\  
&                        Y^+  \ar[rr]^g     && Y.
}
$$

\begin{lemma}\label{lem:pullback}
There is a natural morphism 
$$
g^* \dR f_* C^p_X   \longrightarrow \dR f_*^+ \big(  (g')^* C^p_X\big).
$$
Moreover, for any line bundle $\sM$ on $X$, we similarly have a morphism 
$$
g^* \dR f_* ( C^p_X \otimes \sM^{-1} )  \longrightarrow  \dR f^+_* \big(  (g')^*  (C^p_X \otimes \sM^{-1})   \big)  .
$$
\end{lemma}

\begin{proof}
By (derived) projection formula, and the fact that $C^p_X$ is locally free, we have 
$$
 \dR \pi_* \sO_{X^+} \otimes \sigma^*C^p_X  \cong  \dR \pi_*( \pi^*\sigma^* C^p_X )  =  \dR \pi_* \big( (g')^*  C^p_X  \big). 
$$
Together with the natural map (adjunction) $\sO_{X'} \to \dR \pi_* \sO_{X^+}$ we thus find 

\begin{equation}\label{eq:ADJU}
\sigma^* C^p_X  \longrightarrow  \dR \pi_*  \big( (g')^* C^p_X  \big)  .
\end{equation} 
By applying $\dR f'_*$ to (\ref{eq:ADJU}) we then get 
$$
\dR f_*'  ( \sigma^* C^p_X ) \longrightarrow  \dR f_*^+  \big(  (g')^* C^p_X \big).
$$
On the other hand, by (derived) base change, and flatness of $g$, we have $\dR f'_* (\sigma^*C^p_X) 
\cong  g^* ( \dR f_* C^p_X )$.

The second assertion in the proposition follows from an identical argument. 
\end{proof}

\begin{assumption} 
From now on we will make the extra assumption that the morphism $g$ is flat.
\end{assumption}

\begin{proposition}\label{prop:functorial}
With Assumption~(\ref{eq:extra}), in the setting of~\ref{HodgeSetUp}, there is a commutative diagram of morphisms of systems 
$$
\xymatrix{
g^*  (\sF, \tau)  \ar[rr]^{g^*\Phi} \ar[d]  &&  g^* (\sE, \theta) \ar[d]  \\
(\sF^+, \tau^+)  \ar[rr]^{\Phi^+}  &&  (\sE^+, \theta^+), 
}
$$
which is an isomorphism over $Y^+\setminus D_{f^+}$ for the  
vertical map on the left. Furthermore, the vertical map on the right is an injection over $Y^+$.
\end{proposition}

\begin{proof}
This is a direct consequence of base change and the functoriality of the construction of 
the systems involved. To see this, we note that there is a commutative diagram 
$$
\xymatrix{
(\mu^+)^* (g')^* \big( C^p_X \otimes \sM^{-1} \big)  \ar[d] \ar[rr]  &&  (g'')^*  C^p_T \ar[d] \\
(\mu^+)^* \big(  C^p_{X^+} \otimes (\sM^+)^{-1} \big)   \ar[rr]  &&  C^p_{T^+},
}
$$
so that, after applying $\dR h^+_*$, by the projection formula we have 
$$
\xymatrix{
\dR f_*^+  \big(  (g')^* ( C_X^p \otimes \sM^{-1} )  \big)  \ar[rr]  \ar[d]  &&  \dR h_*^+ ( g'' )^* C^p_T \ar[d] \\
 \dR f^+_*  \big(  C^p_{X^+}  \otimes (\sM^+)^{-1}  \big)  \ar[rr]  &&   \dR h_*^+ C^p_{T^+}.
}
$$

On the other hand, by Lemma~\ref{lem:pullback} we have 
$$
\xymatrix{
g^*  \big( \dR f_* C^p_X \otimes \sM^{-1}  \big)  \ar[rr] \ar[d] &&  g^* \dR h_* C_T^p   \ar[d]  \\
\dR f_*^+  \big(  (g')^* ( C_X^p \otimes \sM^{-1} )  \big)  \ar[rr]   &&  \dR h_*^+ ( g'' )^* C^p_T .
}
$$
By combining these two last diagram we thus find 
$$
\xymatrix{
 g^* \dR f_* (C^p_X \otimes \sM^{-1})   \ar[rr]  \ar[d]  &&  g^* \dR h_* C^p_T   \ar[d] \\
 \dR f^+_*  \big(  C^p_{X^+}  \otimes (\sM^+)^{-1}  \big)  \ar[rr]  &&   \dR h_*^+ C^p_{T^+}.
}
$$

Existence of the map $g^*(\sF, \tau) \to (\sF^+, \tau^+)$, and its compatibility with 
$g^*(\sE, \theta)\to (\sE^+, \theta^+)$, now follows from the associated long exact 
cohomology sequences and flatness of $g$:
$$
\xymatrix{
 g^*\R^if_* (C^p_X \otimes \sM^{-1})    \ar[rr]  \ar[d]  &&  g^* \R^ih_* C^p_T    \ar[d] \\
   \R^i f^+_*  \big( C^p_X \otimes (\sM^+)^{-1}  \big)   \ar[rr]  &&   \R^i h^+_* C^p_{T^+}  .
}
$$
Furthermore, the assumption that $g$ is flat implies that $g^*(\sF, \tau) \to (\sF^+, \tau^+)$ is an isomorphism 
over $Y^+ \setminus D_{f^+}$. 

Now, let $T'$ be a strong desingularization of $( X^+ \times_X T )$ such that there
is a surjective birational map $\sigma: T^+ \to T'$. Set $h': T' \to Y^+$ to be the induced family and
let $(\sE', \theta')$ be the Hodge bundle for 
the canonical extension of the VHS underlying $h'$. Then,
again by base change, we know that there is a morphism 
\begin{equation}\label{Map}
 g^*(\sE, \theta) \longrightarrow (\sE', \theta'),
  \end{equation}
which is an isomorphism over $Y^+\setminus D_{h'}$. 
The injectivity of (\ref{Map}) across $D_{h'}$ follows from the definition (or functoriality) of 
canonical extensions. 

On the other hand, thanks to Deligne~\cite{DeligneHodgeII} and Esnault-Viehweg~\cite[Lem.~1.5]{revI}, we know that 
$\dR \sigma_*  \Omega^p_{T^+/Y^+} (\log \Delta_{h^+}) \cong \Omega^p_{T'/Y^+}(\log \Delta_{h'})$
(see \cite[4.1.2]{VZ02} for the proof in the relative form).
Therefore, $(\sE^+ , \theta^+) \cong (\sE', \theta')$ which induces the required injection. 
\end{proof}

In the setting of Proposition~\ref{prop:functorial}, let $(\sG, \theta)$ and $(\sG^+, \theta^+)$ be, respectively, 
the image of $(\sF, \tau)$ and $(\sF^+, \tau^+)$ under $\Phi$ and $\Phi^+$.
In particular, for each $i$, we have 
$$
\theta(\sG_i) \subset \Omega^1_Y(\log D_f) \otimes  \sG_{i+1}  \; \;  \; \;
\text{and}  \; \; \; \;
\theta^+(\sG_i^+) \subset \Omega^1_{Y^+}( \log D_{f^+} ) \otimes \sG_{i+1}^+ .
$$

Due to the birational nature of the problems considered in this article, in application, we will be able to delete codimenison two 
subschemes of $Y$ whose preimage under $g$ are also of $\codim_{Y^+}\geq 2$. Therefore, as $g$ is flat, we may assume that 
the torsion free system $(\sG, \theta)$ is locally free. On the other hand, after replacing $(\sG^+, \theta^+)$ by its reflexive hull, 
we may also assume that 
$(\sG^+, \theta^+)$ is reflexive. 

\begin{assumption}\label{assump:free}
The torsion free system $(\sG, \theta)$ is locally free and 
$(\sG^+, \theta^+)$ is reflexive.
 \end{assumption}
 

By Proposition~\ref{prop:functorial} we have a commutative 
diagram of systems 
\begin{equation}\label{eq:maps}
\xymatrix{
g^* (\sG,\theta)  \ar[rr] \ar[d] &&  g^*(\sE, \theta) \ar[d] \\
(\sG^+, \theta^+)  \ar[rr]  &&  (\sE^+, \theta^+),
}
\end{equation}
with all maps being injective over $Y^+$. Furthermore, the morphism 
$$
g^*(\sG, \theta )  \longrightarrow  (\sG^+, \theta^+)
$$
is an isomorphism over the $Y^+\setminus D_{f^+}$.

\medskip

We end this subsection with the following lemmas, which will be 
useful for application in Section~\ref{sect:Section4-MainThm}. 
We will be working in the context of the following setup.

\begin{set-up}\label{SetUpLem}
Let $f: X\to B$ and $f_Z: X_Z \to Z$ be flat projective morphism
with connected fibers of dimension $n$. 
Let $g: Z^+ \to Z$ and $\gamma: Z^+\to B$ be two surjective 
flat morphisms. 
Assume that $\gamma$ is finite. 
All varieties are assumed to be smooth. 
Set $X^+_Z$ and $X'$ to be a strong desingularization of 
the normalization of $X\times_Z Z^+$ and $X\times_B Z^+$, respectively, 
with the naturally induced surjective morphisms 
$f^+ : X^+ \to Z^+$, $g' : X^+_Z \to X_Z$, $f': X'\to Z^+$ and 
$\gamma':X'\to X$. Assume that there is a birational map $X'\to X^+$ 
with $\pi': \wtilde X \to X'$ and $\pi^+: \wtilde X\to X^+$ removing its indeterminacy
and so that $\Delta_{\wtilde f}$ is snc, where $\wtilde f: \wtilde X\to Z^+$ is the induced morphism.
By construction we have $D_{f'}= \Supp(\gamma^*D_f)$ and $D_{f^+} = \Supp(g^*D_{f_Z})$.

Given a line bundle $\sA_Z$ on $Z$, define $\sA_{Z^+}:= g^*\sA_Z$.
Furthermore, set 
$$
\sM: = \Omega^n_{X_Z/Z} (\log \Delta_{f_Z}) \otimes f_Z^* (\sA_Z)^{-1} \; , \;  \sM^+ := (g')^*\sM 
\subseteq \Omega^n_{X^+/Z^+} (\log \Delta_{f^+}) \otimes (f^+)^*\sA_{Z^+}^{-1} ,
$$
$$
\sM': = \Omega^n_{X'/Z^+} (\log \Delta_{f'}) \otimes (f')^* (\sA_{Z^+})^{-1}  .
$$
Assuming that $0\neq s^+\in H^0((\sM^+)^m)$, let $\sigma^+: T^+\to X^+$ be a surjective morphism from a 
smooth quasi-projective variety $T^+$ associated to $s^+$, that is $h^0\big((\sigma^+)^* \sM^+\big) \neq 0$. 
Assume further that $\Delta_h$ is snc, where $h^+: = f^+\circ  \tau^+$.
Denote a strong desingularization of $\wtilde X \times_{X^+} T^+$ by $\wtilde T$.
 Set $\widetilde \sigma: \wtilde T \to \wtilde X$ to be the induced morphism and 
 assume that $\Delta_{\wtilde h}$ is snc, where $\wtilde h:= \tilde f \circ \wtilde \sigma$, 
all fitting in the commutative diagram:
$$
\xymatrix{
&&&& \wtilde T  \ar[dl]_{\wtilde \sigma} \ar[dr]  \\
&&& \wtilde X \ar[ld]_{\pi'} \ar[dr]^{\pi^+}  &&  T^+ \ar[dl]_{\sigma^+} \\
X  \ar[d]^f  &&  \ar[ll]_{\gamma'}  X'  \ar[dr]_{f'}    &&   \ar@{-->}[ll]^{\text{birational}}     X^+ \ar[dl]^{f^+}  \ar[rr]^{g'}  && X_Z \ar[d]^{f_Z} \\
 B  &&&  \ar[lll]_{\gamma} Z^+ \ar[rrr]^g  &&& Z.
}
$$

Set $\wtilde f:= \pi^+\circ f^+$.
Define $\overline \sM:= \Omega^n_{\wtilde X/Z^+}(\log \Delta_{\wtilde f}) \otimes (\wtilde f)^* \sA_{Z^+}^{-1}$. 
The natural inclusion $(\pi^+)^* \sM^+ \subseteq \overline \sM$ (after raising the power to $m$) 
identifies a global section $\overline s$ of $\overline{\sM}^m$, determined by 
$s^+$. In particular the induced map $(\wtilde\sigma)^* (\overline \sM)^{-1}\to \sO_{\wtilde T}$ 
factors through $(\wtilde \sigma)^* (\pi^+)^*(\sM^+)^{-1} \to \sO_{\wtilde T}$.
Using the notations in Subsection~\ref{HodgeSetUp}, this implies that 
$$
\xymatrix{
(\wtilde\sigma)^* \big( C^p_{\wtilde X} \otimes (\overline \sM)^{-1}  \big) \ar[rr]  \ar[d]  &&    C^p_{\wtilde T}   \\
(\wtilde \sigma)^*  \big(  C^p_{\wtilde X}  \otimes (\pi^+)^* (\sM^+)^{-1} \big)    \ar[rru]
}
$$
commutes. 

Furthermore, the inclusion $(\pi')^* \sM' \subseteq \overline \sM$ is an equality over $\wtilde X \backslash \Exc(\pi')$. 
Thus, since $\sM'$ is invertible and $X'$ is smooth, 
the section $\overline s\in H^0({\overline\sM}^m)$ induces $s'\in H^0((\sM')^m)$
such that 
\begin{equation}\label{EQUAL}
(\pi')^* s'|_{\wtilde X\backslash \Exc(\pi')}  =  \overline s|_{\wtilde X \backslash \Exc(\pi')} .
\end{equation}
Let $\widehat \sigma: \widehat T\to \wtilde X$ denote cyclic covering
associated to $(\pi')^*s'$. 
Using (\ref{EQUAL}) and the construction of such coverings (\cite[pp. 243--244]{Laz04-I})
$\widehat T$ and $\wtilde T$ are generically isomorphic. 
As such, after replacing $\widehat T$ by a higher birational model, we may assume that 
$\widehat T$ is smooth and $\widehat \sigma$ factors through $\wtilde \sigma: \wtilde T \to \wtilde X$
via a generically finite (in fact birational) morphism $\rho: \widehat T \to \wtilde T$.  
Let $\eta: \widehat T \to T^+$ and $\widehat h: \widehat T\to Z^+$ denote the naturally induced maps.

With the above construction we observe that, over the complement of $\Exc(\pi')$, the two injections 
$\widehat \sigma^* (\overline \sM)^{-1} \to \sO_{\widehat T}$ and $\widehat \sigma^* \big( (\pi')^* (\sM')^{-1} \big) \to \sO_{\widehat T}$
coincide, implying that: 

\begin{observation}\label{OBSS}
The two naturally defined injections 
$$
(\widehat \sigma)^* \big( C^p_{\wtilde X} \otimes (\overline \sM)^{-1} \big) \hooklongrightarrow  \rho^*C^p_{\wtilde T} \subseteq C^p_{\widehat T} 
\;\; ,  \;\;    (\widehat \sigma)^* \big( C^p_{\wtilde X} \otimes ( (\pi')^* \sM' )^{-1} \big) \hooklongrightarrow C^p_{\widehat T} 
$$
coincide over the complement of $\Exc(\pi')$.
\end{observation}

Following the constructions in Subsection~\ref{HodgeSetUp}, let $(\sF^+, \tau^+)$
and $(\wtilde{\sF^+}, \wtilde{\theta^+})$ be the 
logarithmic systems associated to the short exact sequences 
$C^p_{X^+} \otimes (\sM^+)^{-1}$ and $C^p_{\wtilde X} \otimes (\pi^+)^*(\sM^+)^{-1}$.
In particular we have 
$$
 \sF_i^+  =  R^if^+_* \big(  \Omega^{p-i}_{X^+/Z^+} (\log \Delta_{f^+}) \otimes (\sM^+)^{-1}  \big) \;, \;   
 \wtilde{\sF_i^+}  =  R^i \wtilde f_* \big(  \Omega^{p-i}_{\wtilde X/Z^+} (\log \Delta_{\wtilde f}) \otimes (\pi^+)^*(\sM^+)^{-1}  \big)    .
$$
Similarly, define 
$(\sF', \tau')$ and 
$(\wtilde{\sF'}, \wtilde{\tau'})$ to be the logarithmic systems respectively associated to 
$C^p_{X'} \otimes (\sM')^{-1}$ and $C^p_{\wtilde X} \otimes (\pi')^*(\sM')^{-1}$.

Let $(\sE^+, \theta^+)$ and $(\widehat \sE, \widehat \theta)$ be the logarithmic system of Hodge bundles 
associated to the canonical extension of the $\bC$-VHS of weight $n$ underlying the 
smooth loci of $h^+$ and $\widehat h$,  
and denote $(\wtilde \sE, \wtilde \theta)$ to be the image of the system associated to $C^p_{\wtilde T}$ in 
$(\widehat \sE, \widehat \theta)$, induced naturally 
by $\rho^*$. 

Let $\Phi^+: (\sF^+, \tau^+) \to (\sE^+ , \theta^+)$ 
and $\wtilde{\Phi^+}:  (\wtilde{\sF^+}, \wtilde{\tau^+} )\to (\wtilde \sE, \wtilde \theta) 
\subseteq (\widehat \sE, \widehat \theta)$ be the morphism 
of systems defined as in (\ref{eq:KeyMap}) and Proposition~\ref{prop:TajMap}. 
Denote their images respectively by $(\sG^+, \theta^+)$ and $(\wtilde{\sG^+}, \widehat \theta)$.
Furthermore, let $\Phi_{\pi^+}: (\sF^+, \tau^+) \to (\wtilde{\sF^+}, \wtilde{\tau^+})$ 
and $\Phi_{\eta}: (\sE^+, \theta^+) \to (\widehat \sE, \widehat \theta)$ 
be morphisms of systems naturally defined by pullback morphisms $(\pi^+)^*$ and $\eta^*$.
Similarly define $\Phi_{\pi'}: (\sF', \tau') \to (\wtilde{\sF'}, \wtilde{\tau'})$ and 
$\wtilde{\Phi'}:(\wtilde{\sF'}, \wtilde{\tau'})\to (\widehat \sE, \widehat \theta)$, 
with the image of the latter being denoted by $(\wtilde{\sG'}, \widehat \theta)$.
We summarize and further refine these constructions in the following lemma.
\end{set-up}

\begin{lemma}\label{PrelimLem}
In the setting of Set-up~\ref{SetUpLem} we have:
\begin{enumerate}
\item \label{1Comm}
There is a commutative diagram of systems
$$
\xymatrix{
(\sF^+, \tau^+)  \ar[rr]^{\Phi^+}  \ar[d]_{\Phi_{\pi^+}}  &&   (\sE^+, \theta^+)  \ar[d]^{\Phi_{\eta}}  \\
(\wtilde{\sF^+}, \wtilde{\tau^+})  \ar[rr]^{\wtilde{\Phi^+}}  &&    (\widehat \sE, \widehat \theta)  .
}
$$
In particular we have $(\sG^{++},  \widehat \theta)  \subseteq   ( \wtilde{\sG^+} , \widehat \theta) \subseteq  (\wtilde \sE, \wtilde \theta)
\subseteq (\widehat \sE, \widehat \theta)$,
where $(\sG^{++}, \widehat \theta)$ is the image of $(\sF^+, \tau^+)$ under $\Phi_{\eta}\circ \Phi^+$. 

\item  \label{2Include} $\sA_{Z^+} \hookrightarrow   \sG^{++}_0  = \wtilde{\sG^+_0}$.

\item \label{3Sim} There are natural morphisms $(\sF', \tau') \to (\wtilde {\sF'}, \wtilde{\tau'}) \to (\wtilde \sE, \widehat \theta)$. 
Denote the image of $(\sF', \tau')$ in $(\widehat \sE, \widehat \theta)$ by $(\sG', \widehat \theta)$ 
and that of $(\wtilde{\sF'}, \wtilde{\tau'})$ in $(\widehat \sE, \widehat \theta)$ by $(\wtilde{\sG'}, \widehat \theta)$. 
We have $\sG'_0 = \wtilde{\sG'_0} \cong \sA_{Z^+}$.
\end{enumerate}
\end{lemma}

\begin{proof}
Item \ref{1Comm} simply follows from the constructions in Set-up~\ref{SetUpLem} and the functorial properties of the morphisms 
in this diagram (remembering that as in  (\ref{eq:KeyMap})  all are naturally defined by pullback maps). 
More precisely, setting $\sigma:= \eta\circ \sigma^+$, we note that the morphisms 
$$
\xymatrix{
(\sF^+, \tau^+) \ar[r]^(.4){\Phi^+} &  (\sG^+, \theta^+) \subseteq (\sE^+, \theta^+)   \ar[rr]^{\Phi_{\eta}}  &&  (\widehat \sE, \widehat \theta) 
}
$$
are naturally defined by the pullback maps:
$$
\xymatrix{
\sigma^* \big( \Omega^p_{X^+/Z^+} (\log \Delta_{f^+})  \otimes (\sM^+)^{-1}  \big)  \ar[r]  &   \eta^*\Omega^p_{T^+/Z^+}(\log \Delta_{h^+})  
\ar[rr] &&   \Omega^p_{\widehat T/Z^+} (\log \Delta_{\widehat h})  . 
}
$$
As such their composition, which we denote by $\Psi$, satisfies the following claim (Item~\ref{1Comm}). 
\begin{claim}\label{claim:VIA}
$\Psi$ factors through $\wtilde{\Phi^+}$ via $\Phi_{\pi^+}$.
\end{claim}
\noindent
\emph{Proof of Claim~\ref{claim:VIA}.}
By construction we know that $\sigma^*\big(  C^p_{X^+} \otimes (\sM^+)^{-1} \big)  
\to   C^p_{\widehat T}$ 
factors through 
$\wtilde \tau^*\big(C^p_{\wtilde X} \otimes (\pi^+)^* (\sM^+)^{-1}  \big) 
\to C^p_{\widehat T}$. 
After applying $\dR \wtilde \tau_*$ we thus find the following commutative 
diagram of triangles in $D(\widehat X)$: 
\begin{equation}\label{eq:TRAINGLESComm}
\xymatrix{
(\pi^+)^* \big( C^p_{X^+} \otimes (\sM^+)^{-1}  \big) \otimes \dR \widehat \sigma_* \sO_{\widehat T}   \ar[dr]     \ar[rr]  &&   \dR \wtilde \tau_* C^p_{\widehat T}    \\
 &    C^p_{\wtilde X} \otimes (\pi^+)^* (\sM^+)^{-1} \otimes \dR \widehat \tau_* \sO_{\widehat T}   \ar[ur]  .
}
\end{equation}
On the other hand, the diagram 
$$
\xymatrix{
(\pi^+)^* \big( C^p_{X^+}  \otimes (\sM^+)^{-1} \big) \ar[d]  \ar[r]  &   
   (\pi^+)^* \big( C^p_{X^+} \otimes (\sM^+)^{-1}  \big) \otimes \dR \widehat \sigma_* \sO_{\widehat T}    \ar[d]  \\
C^p_{\wtilde X} \otimes (\pi^+)^* (\sM^+)^{-1}   \ar[r]  &  C^p_{\wtilde X} \otimes (\pi^+)^* (\sM^+)^{-1} \otimes \dR \widehat \sigma_* \sO_{\widehat T}   .
}
$$
naturally commutes. From (\ref{eq:TRAINGLESComm}) it thus follows that
\begin{equation}\label{eq:Comm2}
\xymatrix{
(\pi^+)^* \big( C^p_{X^+} \otimes (\sM^+)^{-1}  \big)  \ar[dr]     \ar[rr]  &&   \dR \widehat \sigma_* C^p_{\widehat T}    \\
 &    C^p_{\wtilde X} \otimes (\pi^+)^* (\sM^+)^{-1}    \ar[ur]  .
}
\end{equation}
After applying $\dR \pi^+_*$ to (\ref{eq:Comm2}) we then get
$$
\xymatrix{
 C^p_{X^+} \otimes (\sM^+)^{-1}   \ar[dr]     \ar[rr]  &&   \dR  \sigma_* C^p_{\widehat T}    \\
 &    \dR \pi^+_*\big( C^p_{\wtilde X} \otimes (\pi^+)^* (\sM^+)^{-1} \big).   \ar[ur]  
}
$$
The claim follows from applying $\dR f^+_*$ to this latter commutative diagram. \qed

Items~\ref{2Include} and~\ref{3Sim} similarly follow from the constructions in Set-up~\ref{SetUpLem}.
\end{proof}

\begin{lemma}\label{Claim0}
In the situation of Set-up~\ref{SetUpLem} and Lemma~\ref{PrelimLem} we have 
$\sA_{Z^+}\cong \sG'_0 \subseteq  \sG^{++}_0 \subseteq \widehat \sE_0$.
\end{lemma}

\begin{proof}
This is a direct consequence of the constructions in (\ref{SetUpLem}) and Lemma~\ref{PrelimLem}. 
That is, we consider the auxiliary system $( \overline \sM, \overline \tau )$ 
associated to $C^p_{\wtilde X} \otimes (\overline \sM)^{-1}$, where 
$\overline \sM: =  \Omega^n_{\wtilde X/Z^+}(\log \Delta_{\wtilde f}) \otimes (\pi^+)^*(f^+)^* \sA_{Z^+}^{-1}$
and denote its image under the natural map 
$$
 \overline\Phi:   (\overline \sF, \overline \tau)  \longrightarrow  (\wtilde \sE, \wtilde \theta)\subseteq (\widehat \sE, \widehat \theta)
$$
by $(\overline \sG, \widehat \theta)$.

\begin{claim}\label{claim:KEYKEY}
We have $  (\overline \sG, \widehat \theta)  \subseteq  ( \wtilde{\sG^+} , \widehat \theta ) \subseteq (\wtilde\sE, \wtilde\theta) 
\subseteq (\widehat \sE, \widehat\theta)$
and $\overline \sG_0 = \wtilde{\sG'_0}$.
\end{claim}

\noindent
\emph{Proof of Claim~\ref{claim:KEYKEY}.} 
By constructions in Set-up~\ref{SetUpLem} and using the inclusion $(\pi^+)^*\sM^+ \subseteq \overline \sM$, 
$\overline\Phi$ factors as 
$$
(\overline \sF, \overline \tau) \to  (\wtilde{\sF^+}, \wtilde{\tau^+}) \to  (\wtilde \sE, \wtilde\theta) \subseteq  (\widehat \sE, \widehat \theta) ,
$$ 
which establishes the desired inclusions. 

For the equality $\overline\sG_0 = \wtilde{\sG'_0}$, note that by Observation~\ref{OBSS} 
the maps 
$$
\xymatrix{
\overline \sF_0  \ar[rr]^{\overline\Phi_0} \ar[d]^{\cong} &&  \widehat\sE_0    \\
\wtilde{\sF'_0}    \ar[rru]_{\wtilde{\Phi'_0}}
}
$$
commute. Therefore,  the image of $\overline\sF_0$ and $\wtilde{\sF'_0}$ in $\widehat \sE_0$ coincide. \qed

Now, by Lemma~\ref{PrelimLem} we have $\sG^{++}_0 = \wtilde{\sG^+_0}$ and 
$\sG'_0 = \wtilde{\sG'_0}$. 
On the other hand, by Claim~\ref{claim:KEYKEY} we have 
$$
\wtilde{\sG^+_0}  \cap  \wtilde{\sG'_0} = \overline\sG_0  = \wtilde{\sG'_0} = \sG'_0.
$$
Therefore, $\sG^{++}_0 \cap \sG'_0 = \sG'_0 \cong \sA_{Z^+}$.
\end{proof}

The next lemma helps with identifying a certain subsystem of $(\sG^+, \theta^+)$ (as in Set-up~\ref{SetUpLem}) 
which will be constructed in Proposition~\ref{Prop:GoodSS} (see also \cite[Thm. 3.5.1]{WW20}). 

\begin{lemma}\label{lem:copy}
Given $\Phi_{\eta}: (\sE^+, \theta^+)\to (\widehat \sE, \widehat \theta )$ as in Set-up~\ref{SetUpLem}, 
there is Higgs subsheaf $(\overline \sE, \overline \theta) \subseteq (\sE^+, \theta^+)$ with the following properties. 
\begin{enumerate}
\item \label{ResInject} $\Phi_{\eta}|_{(\overline \sE, \overline\theta)}$ is injective. 
Denote the image of $(\overline \sE, \overline \theta)$ under $\Phi_{\eta}$ by $(\sE'', \widehat \theta)$.

\item \label{Includes}  $\sE^+_0  \subseteq \overline\sE$ and thus, as 
$\sG^+_0 \subseteq \sE_0^+$, we have $\sG^{++}_0\subseteq \sE''$.
\end{enumerate}
\end{lemma}

\begin{proof}
Let $(\mathcal V^+, \nabla^+)$ and $(\widehat{\mathcal V}, \widehat \nabla)$ be the two flat logarithmic connections 
underlying $(\sE^+, \theta^+)$ and $(\widehat \sE, \widehat \theta)$, respectively, 
and $\phi_{\eta}: (\mathcal V^+, \nabla^+) \to (\widehat{\mathcal V}, \widehat \nabla)$
the morphism of holomorphically flat bundles corresponding to $\Phi_{\eta}$.
Set $Z_0\subseteq Z^+$ to be the maximal open subset over which both $\mathcal V^+$ and $\widehat{\mathcal V}$
are polarized $\mathbb C$-VHSs defined by the smooth loci of $h^+$ and $\widehat h$, respectively. By 
Deligne~\cite[Prop.~1.13]{Deligne87} over $Z_0$ both $\mathcal V^+$ and $\widehat{\mathcal V}$ 
are semisimple. Let $\mathcal W^0\subseteq \mathcal V^+|_{Z^0}$ be the smallest 
direct sum of simple summands that contains $h^+_*\Omega^n_{T^+/Z^+} (\log \Delta_{h^+})|_{Z^0} \subseteq \mathcal V^+|_{Z^0}$,
remembering that $h^+_*\Omega^n_{T^+/Z^+}(\log \Delta_{h^+})$ 
is the extension of the lowest piece of the Hodge filtration (and as such is contained in $\mathcal V^+$). 

\begin{claim}\label{claim:FlatInject}
$\Phi_{\eta}|_{\mathcal W^0}: \mathcal W^0 \to \widehat{\mathcal V}|_{Z^0}$ is an injection. 
\end{claim}

\noindent
\emph{Proof of Claim~\ref{claim:FlatInject}.}
Note that there is a natural injection 
$$
h^+_*\Omega^n_{T^+/Z^+}(\log \Delta_{h^+}) \hookrightarrow \widehat h_*\Omega^n_{\widehat T/Z^+} (\log \Delta_{\widehat h}), 
$$
(which is an isomorphism as $\eta$ is birational) so that by the construction of $\Phi_{\eta}$ (or $\phi_{\eta}$) we have the commutative diagram 
$$
\xymatrix{
h_*^+ \Omega^n_{T^+/Z^+}(\log \Delta_{h^+})|_{Z^0}  \ar@{^{(}->}[rr]^{\phi_{\eta}}  \ar[d]^{\subseteq} &&  
  \wtilde h_*\Omega^n_{\wtilde T/Z^+} (\log \Delta_{\wtilde h})|_{Z^0}  \ar[d]^{\subseteq}  \\
\mathcal W^0  \ar[rr]^{\phi_{\eta}}  &&   \widehat{\mathcal W^0} \subseteq \widehat{\mathcal V}|_{Z^0}  ,
}
$$
where $\widehat{\mathcal W^0}$ is the image of $\phi_{\eta}|_{\mathcal W^0}$; again a semisimple 
flat bundle. Now, if $\phi_{\eta}|_{\mathcal W^0}$ is not injective, then $\widehat{\mathcal W^0}$ 
identifies with a proper summand of $\mathcal W^0$. 
In particular $h^+_*\Omega^n_{T^+/Z^+}(\log \Delta_{h^+})|_{Z^0}$ is  
contained in a smaller direct sum of simple summands of $\mathcal V^+|_{Z^0}$ than those forming 
$\mathcal W^0$, contradicting the minimality assumption on the latter. \qed

Now, according to the fundamental result of Jost-Zuo~\cite{JZ97}, 
over smooth quasi-projective varieties, there is an equivalence of categories between semisimple local systems and 
tame harmonic bundles. 
Therefore, $\mathcal W^0$ underlies 
a tame harmonic bundle $(\sE^0, \theta^0)$ (with in particular a structure of a Higgs bundle) 
over $Z^0$ (see also Mochizuki~\cite{MochKH}). Moreover, by the construction of $(\sE^0, \theta^0)$, using 
the above equivalence of categories and 
the fact that $(\sE^+, \theta^+)$ is the canonical extension of a 
tame harmonic bundle over $Z^0$~\cite{Ste76},~\cite[Sect.~22.1]{MochBookII}, 
$(\sE^0, \theta^0) \subseteq (\sE^+, \theta^+)|_{Z^0}$ is a direct summand. 
Further, as a tame harmonic bundle, $(\sE^0, \theta^0)$ extends to a logarithmic Higgs 
bundle $(\overline \sE, \overline\theta)$ on $Z^+$~\cite[Sect.~22.1]{MochBookII} (after removing some subscheme of $Z^+$
of $\codim_{Z^+}\geq 2$ if necessary), whose eigenvalues of the associated residue map 
are contained in $[0,1)$. 
By uniqueness of the canonical extension (and its construction) it follows that $(\overline \sE, \overline \theta)$
is also a direct summand of $(\sE^+, \theta^+)$. 

On the other hand, from Claim~\ref{claim:FlatInject}, and again using the above equivalence of categories, 
it follows that $\Phi_{\eta}|_{\sE^0}: (\sE^0, \theta^0) \to (\widehat \sE, \widehat \theta)|_{Z^0}$ is injective. 
As $\overline \sE$ is torsion free, we find that $\Phi_{\eta}|_{\overline \sE}$ must be injective, verifying Item~\ref{ResInject}.

For Item~\ref{Includes}, by the construction of $\mathcal W^0$ and $\sE^0$, the bundle $\sE^0$ contains 
$\sE^+_0|_{Z^0} = h^+_*\Omega^n_{T^+/Z^+}(\log \Delta_{h^+})|_{Z^0}$. 
On the other hand, we know that $\overline \sE\subseteq \sE^+$ is a direct summand. 
Therefore, as $\sE^+_0$ is torsion free, the naturally defined map 
$\sE^+_0 = h^+_*\Omega^n_{T^+/Z^+}(\log \Delta_{h^+})\to \overline \sE$ 
is indeed an inclusion.
For the rest, note that by Lemma~\ref{PrelimLem} the inclusion $\sG^{++}_0 \subseteq \widehat \sE_0$
factors through $\Phi_{\eta}:\sE^+_0 \to \widehat \sE_0$ and therefore 
by applying $\Phi_{\eta}$ to $\sG^+_0 \subseteq \sE^+_0 \subseteq \overline \sE$ 
we find $\sG^{++}_0 \subseteq \Phi_{\eta}(\overline \sE) = \sE''$.
\end{proof}

\begin{proposition}\label{Prop:GoodSS} 
In the situation of Set-up~\ref{SetUpLem}, assume further that $f_Z$ is semistable and that
$$
\sA_Z\cong  \Big( \det \big((f_Z)_*  \omega^m_{X_Z/Z}\big) \Big)^{N} (-D_{Z})  ,
$$
for some $N\in \bN$ and $D_Z\geq 0$.
Then, we can find a $\gamma^*\Omega^1_B(\log D)$-valued subsystem 
$$
(\sG'', \theta^+) \subseteq  (\sG^+, \theta^+)  \subseteq  ( \sE^+, \theta^+)
$$
equipped with an isomorphism $\sA_{Z^+}\cong \sG''_0$. 
Furthermore, over $Z^+\backslash D_{f^+}$, $(\sG'', \theta^+)$ is also $(g^*\Omega^1_Z)$-valued, 
i.e., we  have $\theta^+( \sG'' |_{Z^+\backslash D_{f^+}} ) 
\subseteq \big(  g^*\Omega^1_Z  \cap  \gamma^*\Omega^1_B(\log D_f) \big)|_{Z^+\backslash D_{f^+}} \otimes \sG''|_{Z^+\backslash D_{f^+}}$. 
\end{proposition}

\begin{proof}
We first make the following observation. 
\begin{claim}\label{claim:value}
$(\sF', \tau')$ (as defined in Set-up~\ref{SetUpLem}) is $\big(\gamma^*\Omega^1_B(\log D)\big)$-valued.
\end{claim}

\noindent
\emph{Proof of Claim~\ref{claim:value}.}
Since $f_Z$ is semistable and $g$ is flat, by~\cite[Sect.~3]{Viehweg83} 
we have $g^* (f_Z)_* \omega^m_{X_Z/Z} = (f^+)_* \omega^m_{X^+/Z^+}$.
Therefore, 
we find 
$g^*\det \big( (f_Z)_*  \big) \omega^m_{X_Z/Z} = \det (f^+)_* \omega^m_{X^+/Z^+}$, 
i.e., 
$$
\sA_{Z^+} \cong (\det f^+_*\omega^m_{X^+/Z^+})^{N}(-D_{Z^+}),
$$
for some $D_{Z^+}\geq 0$.

Let us first assume that $D_{Z^+}=0$.
Consider the system $(\sF_B,  \tau_B)$ on $B$ defined by $C^p_X\otimes  \sM_B^{-1}$, 
where $ \sM_B:= \Omega^n_{X/B} (\log \Delta_f)\otimes f^*(\sA_B)^{-1}$, with $\sA_B := (\det f_*\omega^m_{X/B})^N$.

\begin{subclaim}\label{SC}
We have 
$$
(\gamma')^* \big( \sM_B|_{X\backslash \Delta_f}  \big)  \cong  \sM' |_{X'\backslash \Delta_{f'}}  .
$$
\end{subclaim}
\noindent
\emph{Proof of Subclaim~\ref{SC}.}
Since $f$ is smooth over $X\backslash D_f$ (and thus $D_{f'} = \Supp(\gamma^*D_f)$) it suffices to show that the isomorphism 
\begin{equation}\label{eq:PB}
(\gamma')^* f^* ( \sA_B |_{X\backslash \Delta_f} ) \cong (f')^* \sA_{Z^+}|_{X'\backslash \Delta_{f'}}
\end{equation}
holds. On the other hand, by construction we have 
$f'_*\omega^m_{X'/Z^+} = \wtilde f_* \omega^m_{\wtilde X/Z^+} = f^+_*\omega^m_{X^+/Z^+}$. 
After taking the determinant we therefore find 
\begin{equation*}\label{Fav}
\sA_{Z^+}\cong (\det f'_*\omega^m_{X'/Z^+})^N  .
\end{equation*}
Moreover, by flat base change we have 
$$
\underbrace{(\gamma')^*f^*}_{(f')^* \gamma^*}  (f_*  \omega^m_{X/B} ) |_{X'\backslash \Delta_{f'}} \cong 
    (f')^*  f'_* \omega^m_{X'/Z^+} |_{X'\backslash \Delta_{f'}} .
$$
After removing a subset of $\codim_{Z^+}\geq 2$ and taking the determinant 
we find the desired isomorphism in the subclaim. \qed

Thus, according to Proposition~\ref{prop:functorial}
we have $\gamma^*(\sF_B, \tau_B)|_{Z^+\backslash D_{f'}} \cong (\sF', \tau')|_{Z^+\backslash D_{f'}}$, 
which establishes the claim. 

Now, assume that $D_{Z^+}\neq 0$.
As $\gamma$ is finite, it suffices to establish the claim over $B\backslash D_f$. Therefore, we may assume
that $f'$ and $f$ are smooth. 
With $\sA_B= (\det f_*\omega^m_{X/B})^N$ there is a natural injection $g^*\sA_Z = \sA_{Z^+} \hookrightarrow \gamma^* \sA_B$
from which it follows that $(f')^* \sA_{Z^+}  \hookrightarrow (\gamma')^*f^*\sA_B$. This implies that
$$
(\sM')^{-1} \hookrightarrow  (\gamma')^* \sM_B^{-1}  .
$$
Using the construction in Subsection~\ref{HodgeSetUp} it then follows that there is an injection 
$(\sF', \tau')  \hookrightarrow  \gamma^*(\sF_B, \tau_B)$, proving the claim.
\qed

Now, set 
$$
\sG^{\star}_i:= ( \sG_i^{++} \cap \sG'_i\cap \sE'' ) \subseteq  \sG_i^{++}\cap \sE'' \subseteq  \widehat \sE_i 
$$
so that $(\sG^{\star} = \bigoplus \sG^{\star}_i, \widehat \theta)$ is a subsystem of both $(\sG',\widehat \theta)$ 
and $(\sG^{++}, \widehat \theta)$.
In particular we have
\begin{align*}
  \sG^{\star}_0  =  & \;   ( \sG_0^{++} \cap \sG'_0\cap \sE'') \subseteq \widehat \sE_0 \\
       =  & \;  \sG^{++}_0\cap \sG'_0   ,  \;\;\text{since $\sG^{++}_0 \subseteq \sE''$ by~\ref{Includes}}  \\
       =  & \;  \sG'_0 \cong \sA_{Z^+}   ,  \;\; \text{by Lemma~\ref{Claim0}} .
\end{align*}
As the subsystem $(\sG^{\star}, \widehat \theta) \subseteq (\sG^{++}, \widehat\theta)$ is a 
Higgs subsheaf of $(\sE'', \widehat \theta)$, by Lemma~\ref{lem:copy}, 
there is a $\Phi_{\eta}$-induced isomorphic subsystem $(\sG'', \theta^+)$ of $(\sG^+,\theta^+)$ 
which is a Higgs subsheaf of $(\overline \sE, \theta^+)$.
In particular we have $\sG''_0\cong \sG^{\star}_0 \cong \sA_{Z^+}$.
Moreover, since $(\sG', \widehat \theta)$ and thus $(\sG^\star, \widehat\theta)$ 
is $\big(\gamma^*\Omega^1_B(\log D)\big)$-valued, so is $(\sG'', \theta^+)$.

Furthermore, according to Proposition~\ref{prop:functorial} we have the 
isomorphism of systems
$$
\big(  g^*(\sG, \theta) \big)|_{Z^+\backslash D_{f^+}}  \cong  ( \sG^+, \theta^+ )|_{Z^+\backslash D_{f^+}} .
$$
Let $(\sG''', \theta)$ be the subsystem of $\big( g^* (\sG, \theta) \big)|_{Z^+\backslash D_{f^+}}$
induced by $(\sG'', \theta^+)$ via this isomorphism.
Clearly the isomorphism $ (\sG'', \theta^+ )|_{Z^+\backslash D_{f^+}} \cong (\sG''', \theta)$ 
implies that $(\sG'', \theta^+ )|_{Z^+\backslash D_{f^+}}$ is $(g^*\Omega^1_{Z\backslash D_f})$-valued. 
\end{proof}

\medskip

\subsection{Functoriality.~II: descent of kernels.}
For the purpose of application later on in Section~\ref{sect:Section4-MainThm}, we need to further 
refine our understanding of the 
properties of the systems constructed in Subsection~\ref{subsect:functorial}, when 
$g$ is induced by a \emph{flattening} of a proper morphism, cf.~\cite{GR71}.
To this end, we consider
the following situation. 

Let $f: V\to W$ be a projective morphism of smooth quasi-projective varieties with 
connected fibers of positive dimension. Let $f': V'\to W'$ be a desingularization of a flattening 
of $f$, with the associated birational morphisms $\pi: V'\to V$ and $\mu:W'\to W$,
so that, by construction, every $f'$-exceptional divisor is $\pi$-exceptional.

\begin{definition}[codimension one flattening]\label{def:flatten}
In the above setting, let $V^\circ\subseteq V$ be the complement of the center of 
$\pi$. We call the induced flat morphism $f^\circ: V^\circ \to W'$ a \emph{codimension
one flattening of $f$}.
\end{definition}

\smallskip

\begin{notation}
In the rest of this article we denote the
reflexivization of the determinant sheaf by $\det(\cdot)$.
\end{notation}

We will be working in the setting of Set-up~\ref{SetUpLem}.

\begin{notation}\label{notnot}
In the setting of Proposition~\ref{Prop:GoodSS} define 
$\sN''_i:= \ker(\theta^+|_{\sG''_i})$ and $\sN^+_i:= \ker(\theta^+|_{\sG_i^+})$. 
\end{notation}

\begin{proposition}[descent of kernels of subsystems of VHS]\label{prop:descent}
In the setting of Subsection~\ref{SetUpLem}, assume that the varieties are projective and that the maps
exist after removing closed subsets of $Z$ and $B$ of $\codim\geq 2$.
If $g: Y^+ \to Y$ is a codimension one flattening 
of a proper morphism with connected fibers, then, for every $i$, there is a
pseudo-effective line bundle $\sB_i$ on $Z$ such that

\begin{equation}\label{eq:descent1}
\big( (\det \sN^+_i)^{-1} \big)^{a_i} \cong g^*\sB_i ,
 \end{equation}
for some $a_i \in \bN$. 

\end{proposition}

\begin{proof}
After replacing $Y$ and $Y^+$ by $Z$ and $Z^+$ in Proposition~\ref{prop:functorial}, let 
$(\sG, \theta)$ be the image of $\Phi: (\sF, \tau)\to (\sE, \theta)$. 
Set $\sN_i := \ker (\theta|_{\sG_i})$. We first consider the case where $g$ is assumed to be proper.
Again, as $g$ is flat, pre-image of subsets of $Y$ of $\codim_Y \geq 2$ are of $\codim_{Y^+}\geq 2$ 
and therefore we may assume that $\sN_i$ is locally free. 
From Diagram~\ref{eq:maps} (and Proposition~\ref{prop:functorial})
it follows that there is an injection 

$$
g^* \big(\det \sN_i \big)  \longrightarrow  \det (\sN_i^+),
$$
which is an isomorphism over $Z^+\setminus D_{f_Z^+}$. 
Therefore, $\det \sN_i^+$ is $g$-effective and that $\det \sN_i^+ \cong \sO_{Z_z^+}$, 
for a general point $z\in Z$. Thanks to properness and flatness of $g$, from the latter isomorphism it follows that
$$
\det \sN_i^+ \equiv_g 0.
$$
Therefore $\det \sN_i^+$ is trivial over $Z$. Consequently there is a line bundle 
$\sB_i$ on $Z$ satisfying the isomorphism (\ref{eq:descent1}). 

On the other hand, thanks to weak seminegativity of kernels of Higgs fields underlying polarized 
VHS of geometric origin~\cite{Zuo00}
(see~\cite[Sect.~3]{Taj18} for further explanation and references), $(\det \sN^+_i)^{-1}$ 
is pseudo-effective. Therefore so is $\sB_i$, cf.~\cite{BDPP}.

For the case where $g$ is not proper, we repeat the same argument for the flattening of $g$
(from which $g$ arises), after removing the non-flat locus from the base.
\end{proof}

\medskip

Next, we recall the trick of Kov\'acs and Viehweg-Zuo involving iterated Kodaira-Spencer maps,
which we adapt to our setting.

\begin{lemma}\label{lem:KS}
In the setting of Proposition~\ref{Prop:GoodSS},
assume that 
$\kappa( Z, \sG_0 ) = \kappa(Z, \sA_Z)>0 $. Then, up to a suitable power, 
there is an integer $m>0$ for which $\theta^+$ induces an injection

\begin{equation}\label{eq:KS}
  \sG''_0 \otimes \underbrace{\big(   \det \sN_m^+ \big)^{-1}}_{g^*\sB_m}  \hooklongrightarrow  
          \big( \gamma^*\Omega^1_B(\log D_f)  \big)^{\otimes k}     \subseteq   \big(  \Omega^1_{Z^+} (\log D_{f^+})  \big)^{\otimes k},
\end{equation}
for some $k\in \bN$ and where $\sB_m$ is a pseudo-effective line bundle on $Z$. 
Furthermore, over $Z^+\backslash D_{f^+}$ the injection (\ref{eq:KS}) factors through the inclusion 
$$
\big(  g^*\Omega^1_Z \cap  \gamma^* \Omega^1_B(\log D_f) \big)^{\otimes k} |_{Z^+\backslash D_{f^+}} 
 \subseteq  \big(  \gamma^*\Omega^1_B(\log D_f)  \big)^{\otimes k}|_{Z^+\backslash D_{f^+}}  .
$$
\end{lemma}

\begin{proof}
By Proposition~\ref{Prop:GoodSS} we have $\sG''_0\cong \sA_{Z^+}= g^*\sA_Z$ so that $\kappa(\sG''_0)>0$. 
Noting that (again by Proposition~\ref{Prop:GoodSS}) we have 
$(\sG'', \theta^+) \subseteq  (\sG^+, \theta^+)  \subseteq (\sE^+, \theta^+)$, 
for any non-negative integer $i$, we consider the image $\sG''_0$ under the morphism 
$$
 \theta^+_i : =  \underbrace{(\id \otimes  \theta^+) \circ  \ldots \circ (\id \otimes  \theta^+)}_{\text{$i$ times}} \circ \;  \theta^+ : 
 \sG''_0 \longrightarrow  \big(  \gamma^*\Omega^1_B(\log D_f)  \big)^{\otimes (i+1)}  \otimes   \sG''_{i+1}.
$$
Let $m:= \max \{  i  \; \big|  \; { \theta^+}_i( \sG''_0 ) \neq 0 \}$ so that there is an injection 

$$
 \sG''_0   \hooklongrightarrow    \big(  \gamma^*\Omega^1_B(\log D_f)  \big)^{\otimes m}  \otimes  \sN''_m ,
$$
where $\sN''_i: =\ker ( \theta^+|_{\sG''_i} )$ (as in Notation~\ref{notnot}).

\begin{claim}\label{claim:bigm}
$m\geq 1$.
\end{claim}

\noindent
\emph{Proof of Claim~\ref{claim:bigm}.} If the map $ \theta^+:  \sG''_0 \to  \gamma^*\Omega^1_B(\log D_f)  \otimes  \sG''_1$ is zero, 
then $\sG''_0$ is anti-pseudo-effective~\cite{Zuo00}.
But this contradicts the inequality $\kappa(\sG''_0)>0$. \qed

Now, from the inclusion of the systems $(\sG'', \theta^+) \subseteq (\sG^+, \theta^+)$ 
we know that $\sN''_m\subseteq \sN^+_m$ 
(Proposition~\ref{Prop:GoodSS}). 
Therefore, there is an injection 
$$
 \sG''_0   \hooklongrightarrow    \big(  \gamma^*\Omega^1_B(\log D_f)  \big)^{\otimes m}  \otimes  \sN^+_m .
$$
Consequently, we find the desired injection~(\ref{eq:KS}).
The isomorphism involving the pseudo-effective line bundle $\sB_m$ follows from Proposition~\ref{prop:descent}.

The last assertion is the direct consequence of the fact that by Proposition~\ref{Prop:GoodSS}
we have $\theta^+( \sG'' |_{Z^+\backslash D_{f^+}} ) 
\subseteq \big(  g^*\Omega^1_Z  \cap  \gamma^*\Omega^1_B(\log D_f) \big)|_{Z^+\backslash D_{f^+}} \otimes \sG''|_{Z^+\backslash D_{f^+}}$.
\end{proof}

\medskip

\section{A bounded moduli functor for polarized schemes}
\label{sect:Section3-Functor}

In this section we will construct a moduli functor that is especially tailored to the 
study of projective families of good minimal models with canonical singularities
(see~\cite{KM98} and \cite{KollarSingsOfTheMMP} for background on the minimal model program 
and the relevant classes of singularities). Let us first recall a few standard notations and 
definitions. In this section all schemes are assumed to be separated and of finite type (see \cite[p.12]{Viehweg95}). 

Let $X$ be a normal scheme and $K_X$ its canonical divisor. By $\omega_X$ we denote 
the divisorial sheaf $\sO_X(K_X)$. 
For a morphism of normal schemes 
$f:X\to B$, assuming that $K_B$ is Cartier, we set $\omega_{X/B}:= \sO_X (K_{X/B})$, 
where $K_{X/B}:= K_X - f^* K_B$. 
Given a coherent sheaf $\sF$ on $X$ and any $m\in \bN$, we define $\sF^{[m]}: = (\sF^{\otimes m})^{**}$ 
to be the $m$-th reflexive power of $\sF$. 

\begin{definition}[relative semi-ampleness]
\label{def:relsemi}
Given a proper morphism of $f:X\to B$ of schemes and a line bundle $\sL$ on $X$, 
we say $\sL$ is semi-ample over $B$, or $f$-semi-ample, if for some $m\in \bN$ 
the line bundle $\sL^m$ is globally generated over $B$, that is
the natural map $f^* f_*\sL^m  \to \sL^m$ is surjective. 
\end{definition}
We note that from the definition it follows that for $f$-semi-ample $\sL$ we have a naturally
induced morphism 

\begin{equation}
\label{canmap}
\psi: X\longrightarrow  \mathbb P_B(f_* \sL^m):= \mathrm{Proj}_{\sO_B} ( \Sym (f_*\sL^m) )
\end{equation}
over $B$, with a $B$-isomorphism $\sL^m \cong \psi^* \sO_{\mathbb P (f_*\sL)}(1)$. 
In particular $\sL^m|_{X_b}$ is globally generated, for every $b\in B$.
Moreover, we say $\sL$ is $f$-ample, if $\sL$ is $f$-semi-ample and the morphism (\ref{canmap}) is an embedding 
over $B$ (see for example \cite[Sect.~1.7]{Laz04-I} for more details).


\begin{notation}[pullback and base change] For every morphism $\alpha: B'\to B$, we denote 
the fiber product $X\times_B B'$ by $X_{B'}$, with the natural projections $f':X\times_B B' \to B'$ 
and $\pr: X\times_B B'\to X$.
Furthermore, for a coherent sheaf $\sF$ on $X$, we define $\sF_{B'} : = \pr^*\sF$.
\end{notation}

We begin by recalling Viehweg's moduli functor $\mfM$ for polarized schemes~\cite[Sect~1.1]{Viehweg95}.
The objects of this functor are isomorphism classes of projective polarized schemes $(Y, L)$, with $L$ being ample. 
We write $(Y,L)\in \Ob(\mfM)$.
The morphism $\mfM : \mathfrak{Sch}_{\bC} \to 
\mathfrak{Sets}$ is defined by

\begin{equation*}
\begin{aligned}
\mfM (B)  = &  \Big\{ \text{Pairs} \; (f: X\to B, \sL)\;  \bigm\mid  \text{$f$ is flat and projective, $\sL$ is invertible} \\ 
                   & \;\;  \;\;  \;\;  \;\;  \;\;  \;\;   \;\;  \;\;  \;\;  \;\;  \;\;  \;\;  \;\;  \;\;   \;\;  \;\;  \;\;  \left. \text{and $(X_b,\sL_b)\in \Ob(\mfM)$, for all $b\in B$}, \Big\} \right/ \sim ,
\end{aligned}
\end{equation*}
for any base scheme $B$. Here, the equivalence relation $\sim$ is given by 

$$
\begin{aligned}
\big(  f_1: X_1 \to B , \sL_1  \big)  \sim  \big( f_2: X_2 \to B , \sL_2   \big)   \iff & \text{there is a $B$-isomorphism $\sigma: X_1\to X_2$}  \\
        & \text{such that $ \sL_1 \cong \sigma^* \sL_2 \otimes f_1^*\sB$}, \\
        & \text{for some line bundle $\sB$ on $B$}.
\end{aligned}
$$

\begin{definition}[\protect{\cite[Def.~2.2]{Hassett-Kovacs04}}]
\label{def:openclosed}
Let $\mathfrak F \subset \mathfrak M$ be a submoduli functor. We say $\mathfrak F$ is \emph{open}, if 
for every $(f: X \to B, \sL) \in \mfM(B)$ the set 
$V=\{ b \in B \; |\;  (X_b, \sL_b) \in \Ob(\mathfrak F)  \}$ is open in $B$ and $(X_V \to V, \sL_V) \in \mathfrak F(V)$.
The submoduli functor $\mathfrak F\subset \mathfrak M$ is \emph{locally closed}, if   
for every $(f: X \to B, \sL) \in \mfM(B)$, there is a locally closed subscheme $j: B^u \hookrightarrow B$
such that for every morphism $\phi: T\to B$ we have: $(X_T \to T, \sL_T) \in \mathfrak F(T)$, if and only if 
there is a factorization
$$
\xymatrix{
T \ar@/^5mm/[rr]^{\phi} \ar[r] & B^u \ar[r]^{j}  & B. 
}
$$
\end{definition}

We note that by definition $\mathfrak F \subset \mfM$ is open, if and only if it is locally closed 
and $B^u$ is open.

\begin{definition}[\protect{\cite[Def.~1.15.(1)]{Viehweg95}}]
\label{def:bounded}
Given a moduli functor of polarized schemes $\mathfrak F$, by $\mathfrak F_h$ 
we denote the submoduli functor whose objects $(Y,L)$ have $h$ as their Hilbert polynomial with respect to 
$L$. We say a submoduli functor $\mathfrak F_h \subset \mathfrak M_h$ is \emph{bounded}, if there is $a_0\in \bN$ such that, 
for every $(Y,L) \in \Ob (\mathfrak F_h)$ and any $a\geq a_0$, the line bundle $L^a$ 
is very ample and $H^i(Y, L^a)=0$, for all $i>0$.
\end{definition}

For a positive integer $N$, we now consider a new submoduli functor $\mfM^{[N]}\subset \mfM$, whose objects
$(Y,L)$ verify the following additional properties. 

\begin{enumerate}
\item \label{object1} $Y$ has only canonical singularities. 
\item \label{object2} $\omega_Y^{[N]}$ is invertible and semi-ample ($N$ is not necessarily the minimum such integer).
\item \label{object3} For all $a\geq 1$, the line bundle $L^a$ is very ample and $H^i(Y,L^a)=0$, for all $i>0$.
\end{enumerate}

\begin{remark}\label{rk:bounded}
Condition~\ref{object3} means that $\mfM_h^{[N]}$ is bounded
by construction (see Definition~\ref{def:bounded}). 
\end{remark}

We note that, with fibers of $(f: X\to B) \in \mathfrak M(B)$ being normal, if $B$ is nonsingular, then $X$ is also normal. 
The following observation of Koll\'ar shows that over nonsingular base schemes, 
for such morphisms a reflexive power $N$ of $\omega_{X/B}$ is invertible. 
Therefore, over regular base schemes, the formation of $\omega_{X/B}^{[N]}$ commutes with 
pullbacks~\cite[Lem.~2.6]{Hassett-Kovacs04}.
We will see in Subsection~\ref{SS:FT} that this property is crucial 
for $\mathfrak M^{[N]}_h$ to be well-behaved. 


\begin{claim}[\protect{cf.~\cite[Lect.~6]{CKM88}}]\label{claim:Kol}
Let $f:X\to B$ be a flat projective morphism of varieties, with $B$ being smooth.
If $X_b$ has only canonical singularities with invertible $\omega_{X_b}^{[N]}$,  
then $\omega_X^{[N]}$ and thus $\omega^{[N]}_{X/B}$ are invertible near $X_b$. 
Moreover, $\omega_{X/B}^{[N]}$ is flat over a neighborhood of $b$. 
\end{claim}

\noindent
\emph{Proof of Claim.~\ref{claim:Kol}}. For every $x\in X_b$, let $\rho_x: U'_x \to U_x$ be the 
local lift of the index-one covering of $(X_b, x)$ over an open subset $V_x\subseteq B$~\cite[Cor.~6.15]{CKM88} 
so that $\omega_{(U'_x)_b}$ is invertible.
By construction $(U'_x)_b$ has only canonical and therefore rational singularities \cite[Cor.~5.25]{KM98}.  
As rational singularities degenerate into rational singularities~\cite{Elkik78},
$U'_x$ has rational singularities and the 
induced family $f\circ \rho_x: U'_x \to V_x$ has Cohen-Macaulay fibers (after restricting to a smaller subset, 
if necessary).
Using base change through $b\to V_x$ we thus find that $(\omega_{U'_x/V_x})_{b}$ is invertible 
\cite[3.6.1]{ConDualBook}. Therefore, so is $\omega_{U'_x/V_x}$. 
Since $V$ 
is regular, it follows that $\omega_{U'_x}$ is also invertible. Consequently $N\cdot K_{U_x}$ is Cartier,
as required. Furthermore, $\omega_{U'_x/V_x}$ is flat over $V_x$, 
and thus so is $(\rho_x)_* \omega_{U'_x/V_x}$. On the other hand, by construction, 
$\omega^{[N]}_{U_x/V_x}$ is a direct summand of $(\rho_x)_* \omega_{U'_X/V_x}$, cf.~\cite[Cor.~3.11]{EV92}.
Therefore, 
$\omega^{[N]}_{Ux/V_x}$ 
is flat over $V_x$.
\qed

\smallskip

\subsection{The Parametrizing space of $\mfM_h^{[N]}$}
\label{SS:FT}
Our aim is now to show that 
the functor $\mfM^{[N]}_h$ has an algebraic coarse moduli space. The next proposition
is our first step towards this goal. For the definition of a separated functor of polarized schemes 
we refer to \cite[Def.~1.15.(2)]{Viehweg95}.

\begin{notation}
For any $d\in \bN$, by $\Hilb^d_h$ we denote the Hilbert scheme of projective subschemes of $\mathbb P^d$ with 
Hilbert polynomial $h$. 
\end{notation}

\begin{proposition}\label{prop:good}

The subfunctor $\mfM_h^{[N]}\subset \mfM_h$ is open (thus locally closed) and separated.
\end{proposition}

\begin{proof}
We first show that $\mfM^{[N]}$ is open (Definition \ref{def:openclosed}). 
This can be done by establishing the openness of each of Properties~\ref{object1}--\ref{object3}
in the following order, assuming that the special fiber $X_b$ is an object of $\mfM^{[N]}$. 
Using base change, we may assume that $B$ is nonsingular (which implies that $X$ is assumed to be 
normal).

\noindent 
\emph{Very ampleness:} Using the vanishing $H^i(X_b,\sL_b)=0$ from 
\cite[Thms.~1.2.17 and 1.7.8]{Laz04-I}, in the very ample case, it follows that 
the morphism $X\to \mathbb P_B (f_*\sL)$ arising from the canonical map 
$f^*f_* \sL \to \sL$ is an immersion along $X_b$ 
and thus an immersion over an open neighborhood of $b$. In particular each $\sL_{b'}$ is very ample
over this neighborhood.

\smallskip
\noindent
\emph{Degeneration of index and singularities:} By Claim~\ref{claim:Kol} we know that  
$\omega^{[N]}_{X/B}$ is
invertible near $X_b$. We also know that nearby fibers are all normal (in fact rational \cite{Elkik78}).
Therefore, by base change, we find that, for every $b'$ near $b$, we have
 $\omega_{X/B}^{[N]}|_{X_{b'}} \cong 
\omega_{X_{b'}}^{[N]}$, showing that the nearby fibers are of index $N$ too.

Now, the fact that $X$ has only canonical singularities near $X_b$
follows from \cite{Kawamata99}, when $\dim B=1$. When $\dim B=2$, we consider 
the normalization of a curve passing through $b$ and use 
inversion of adjunction, cf. \cite[Sect.~5.4]{KM98}. 
For higher dimensions we argue similarly using induction on $\dim B$.

\smallskip

\noindent
\emph{Global generation:} To show that semi-ampleness of the canonical divisor is open (openness of \ref{object2}), 
we note that $\omega_{X/B}^{[N]}$ is invertible and flat over a neighborhood of $b$ (Claim~\ref{claim:Kol}). 
Let $\nu$ be an integer for which $\omega_{X_b}^{[N]\cdot \nu}$ is globally generated.
According to
Takayama~\cite{tak07}, 
the function $b' \mapsto h^0(X_{b'}, \omega_{X_b'}^{[N]\cdot \nu})$ is constant over the open neighborhood of $b$ where 
each $X_{b'}$ has only canonical singularities. Therefore, by 
\cite[Cor.~12.9]{Ha77}, the natural map 

$$
f_*\omega^{[N]\cdot \nu}_{X/B} \otimes \bC(b)  \longrightarrow H^0(X_b, \omega^{[N]\cdot \nu}_{X_b})
$$
is an isomorphism in a neighborhood of $b$. On the other hand, the restriction map 
$H^0( X_b , \omega^{[N]\cdot \nu}_{X_b}  ) \otimes \sO_{X_b} \to \omega^{[N]\cdot \nu}_{X_b}$
is surjective. Therefore, using Nakayama's lemma, we find that the canonical 
map $f^* f_* \omega^{[N]\cdot \nu}_{X/B} \to \omega^{[N]\cdot \nu}_{X/B}$ is surjective 
along $X_b$. It follows that this map is surjective over a neighborhood of $b$, i.e. $\omega_{X/B}^{[N]\cdot \nu}$ 
is globally generated over this neighborhood. 


\smallskip

It remains to verify that $\mfM_h^{[N]}$ is separated. Let $R$ be a discrete valuation ring (DVR for short) and $K$ its field of fractions.
Define $B=\Spec(R)$ and consider two polarized families 

\begin{equation}\label{eq:families}
(f_1:  X_1 \longrightarrow B, \sL_1)  \; \; , \; \;  (f_2: X_2\longrightarrow B, \sL_2) \; \; \in \; \mfM^{[N]}(B),
\end{equation}
that are isomorphic (as families of polarized schemes) over $\Spec(K)$. Let us denote this
isomorphism by $\sigma^\circ: X_1^\circ \to X_2^\circ$, where $X_i^\circ$ denotes the restriction 
of the family $X_i$ to $\Spec(K)$, and, for every $b\in B$, define $X_{i,b}: = (X_i)_b$, $\sL_{i,b}=\sL_i|_{X_{i,b}}$, with $X_{i,0}$
denoting the special fiber. 
Using the two properties in Item~\ref{object3}, for $i=1,2$, as was shown in the very-ampleness case, 
we find that the natural morphisms 
$\psi_i :  X_i  \longrightarrow   \mathbb{P}_B \big(  (f_i)_* \sL_i \big)$ are embeddings over $B$
and  

$$
\sL_i \cong  \psi_i^*  \sO_{ \mathbb{P}_B \big((f_i)_* \sL_i\big) }  (1)  .
$$

\begin{claim}\label{claim:divisors}
In this context, using  the morphism 
$\psi_i: X_i\to \mathbb P_B \big(  (f_i)_* \sL_i \big)$, we can find Cartier divisors

$$
D_i \in |\sL_i | 
$$
such that near the special fibers $X_{i,0}$ we have:
\begin{enumerate}
\item\label{prop1} $D_i$ avoids the generic point of every fiber $X_{i,b}$, 
\item\label{prop3} $ D_1|_{X_1^0} =  (\sigma^\circ)^*  D_2$, and that, 
\item\label{prop2} for some integer $m$, $(X_{i,0}, \frac{1}{m} D_{i,0})$ is log-canonical (lc for short), 
where $D_{i,0}: = D_i |_{X_{i,0}}$.
\end{enumerate}

\end{claim}

\noindent
\emph{Proof of Claim~\ref{claim:divisors}.}
By Item~\ref{object3}, and using the fact that the Hilbert polynomial is fixed in the family, we have 
$h^0(X_{i,b}, \sL_{i,b}) = d+1$, for some $d\in \bN$. 
With $\mfM_h^{[N]}$ being bounded, we have 

$$
\xymatrix{
X_i  \ar[rr]^{\psi_i} \ar[dr]  &&  \mathbb P_B (f_i)_* \sL_i \ar[dl]   \ar[r]  &   \mathbb P^d\times \Hilb_h^d   \ar[d]^{\pr_2}  \\
&    B   \ar[rr] &&            \Hilb^d_h.
}
$$
Since $\psi_i$ is embedding over $B$ and $\mathbb P_{B} (f_i)_* \sL_i$ is fiberwise isomorphic to $\mathbb P^d$ 
we have $\mathbb P_B(f_i)_* \sL_i \cong \mathbb P^d\times B$.

Furthermore, the isomorphism $\sigma^0: (X_1^0, \sL_1|_{X_1^0}) \to (X_2^0, \sL_2|_{X_2^0})$ over $\Spec (K)$ 
naturally induces the $(\Spec K)$-isomorphism 

$$
\sigma^0_{\mathbb P}: \mathbb P_{\Spec K} (f_1)_*  \sL_1^0 \longrightarrow   \mathbb P_{\Spec K} (f_2)_* \sL_2^0
$$
and the resulting diagram 

\begin{equation}\label{HilbDiagram}
\xymatrix{
&  X_1  \ar[dl]_{f_1}  \ar@{-->}[dd]^{\sigma^0}  \ar[rr]^(0.4){\psi_1}  &&    \mathbb P_B (f_1)_* \sL_1 \ar@{-->}[dd]^{\sigma^0_{\mathbb P}}  
                                                                                                     \cong \mathbb P^d \times B  \ar[dr]_{\pr_1}   \\
B                                &&&&            \mathbb P^d                                \\
&  X_2   \ar[ul]^{f_2}     \ar[rr]^(0.4){\psi_2}  &&   \mathbb P_B (f_2)_* \sL_2        \cong    \mathbb P^2 \times B                         \ar[ur]^{\pr_1}  ,
}
\end{equation}
which commutes over $\Spec (K)$. Now, using this construction, including the fact that $\psi_i$ is fiber-wise embedding, 
for a general member $D\in | \sO_{\mathbb P^d}(1) |$, we can ensure that $D_i: = \psi_i^* (\pr_1^* D)$ 
is a divisor on $X_i$ that does not contain the generic point of $X_{i,0}$. 
Moreover, by the commutativity of (\ref{HilbDiagram}) we have $(\sigma^0)^* D_2 = D_1$. 
Finally, let $m$ be sufficiently large so that $(X_{i,0}, \frac{1}{m} D_{i,0})$ is lc.
This finishes the proof of the claim. \qed

\smallskip

Now, using Claim~\ref{claim:Kol}, by inversion of adjunction we find that $(X_i, \frac{1}{m}D_i + X_{i,0})$ is lc and thus, 
by specialization, so is $(X_{i,b}, \frac{1}{m}D_{i,b})$, for a general $b\in \Spec(K)$, where $D_{i,b}: = D_i |_{X_{i,b}}$.
Also, as $D_i$ is fiber-wise very amply, using nefness of $K_{X_i/B}$ we find that 
$K_{X_i/B}+ \frac{1}{m} D_i$ is fiber-wise ample so that, for each $i$, $(f_i: (X_i, \frac{1}{m} D_i) \to B)$ is a \emph{stable family of pairs}, 
cf.~\cite[Def--Thm.~4.7]{KolBook17}. 
 Now, thanks to the separatedness of functors 
of stable families of pairs 
\cite[Thm.~4.1]{KolBook17} over regular base schemes, 
the two families $\big(f_1: (X_1, \frac{1}{m}D_1) \to B\big)$ and $\big(f_2: (X_2, \frac{1}{m} D_2) \to B\big)$ are isomorphic over $B$
(as families of pairs), near the special fiber.
In particular we have 
$$
\big(f_1: X_1 \to B, \sL_1\big)  \sim \big(f_2: X_2 \to B, \sL_2\big) ,
$$
as required. 
\end{proof}







The following proposition is now a consequence of Proposition~\ref{prop:good} and a collection of well-known results 
in the literature. For the definition and basic properties of algebraic spaces we refer to \cite[Chapts.~1,2]{Knu71}.

\begin{proposition}\label{prop:functor}
The moduli functor $\mfM_h^{[N]}$ has an algebraic space of finite type $M_h^{[N]}$ as its
coarse moduli. 
\end{proposition}

\begin{proof}
Using Item~\ref{object3} for every $(Y,L) \in \Ob (\mfM^{[N]}_h)$ we have $h^0(L)=d+1$, for some $d\in \bN$.
Set $\mathcal X' \subset \mathbb P^d \times \Hilb^d_h$ to be the universal object with 

$$
\xymatrix{
\mathcal X'  \ar[dr] \ar@{^{(}->}[r] &  \mathbb P^d \times \Hilb_h^d  \ar[d]^{\pr_2}  \ar[r]^(0.7){\pr_1}  &   \mathbb P^d . \\
&   \Hilb_h^d
}
$$
We note that by Proposition~\ref{prop:good} $\mathfrak M_h$ is bounded and locally closed, 
with $L$ being very ample for 
every $(Y,L) \in \mathfrak M^{[N]}$. Therefore, there is a subscheme $H^u\subset \Hilb^d_h$ such that 
$\mathcal X'$ restricts to the universal family for the associated Hilbert functor of 
embedded schemes in $\mathfrak M_h^{[N]}$ 

$$
\big(   f_{\mathcal X}: \mathcal X \to H^u, \underbrace{(\xi)^* \pr_1^* \sO_{\mathbb P^d}(1)}_{\sL_{\mathcal X}}  \big)  \in \mathfrak M_h^{[N]} (H^u)
$$
via $\xi: \mathcal X \hookrightarrow \mathcal X'$ cf.~\cite[Sect.~1.7]{Viehweg95}.
Now, as $H^u$ is naturally equipped with the action of $\mathbb PG: = \PGL(d+1, \bC)$, following~\cite{Viehweg95}, 
we need to show that the quotient of $H^u$  by $\mathbb PG$ 
is a geometric categorical quotient (see \cite[Def.~2.7]{Kol97} or \cite[Def.~1.8]{KM97} for the definition).
Thanks to \cite[Thm.~1.5]{Kol97}, \cite[Cor.~1.2]{KM97} and \cite[Sect.~7.2]{Viehweg95} it suffices to
establish the following claim.

\begin{claim}\label{claim:proper}
The action $\overline \sigma$ of $\mathbb PG$ on $H^u$ is proper, that is 
the morphism 
$$
\overline \psi: = (\overline \sigma, \pr_2): \mathbb PG \times H^u \longrightarrow  H^u \times H^u 
$$
is proper. Consequently, the action of $G:=\SL(d+1,\bC)$ on $H^u$ is proper with finite stabilizers. 
\end{claim}

\noindent
\emph{Proof of Claim~\ref{claim:proper}.} We follow the arguments of \cite[Lem.~7.6]{Viehweg95}. 
We recall that by the valuative criterion (\cite[Lem.~2.4]{Kol97}) it suffices to show that  
for every DVR $R$, with field of fractions $K$, $B= \Spec(R)$, and any 
commutative diagram 

\begin{equation}
\label{eq:lifting}
\xymatrix{
\Spec (K)  \ar[rr]^{\delta} \ar@{^{(}->}[d]   &&   \mathbb PG\times H^u   \ar[d]^{\overline\psi} \\
B  \ar[rr]^{\tau}  &&   H^u\times H^u  ,
}
\end{equation}
there is an extension $\overline \delta: B \to \mathbb PG\times H$ such that $\overline \delta|_{\Spec K}= \delta$
and $\tau = \overline\psi\circ \overline \delta$. To do so we consider the two families 

$$
(f_i: X_i \to B  , \sL_i) \in \mathfrak M_h^{[N]}(B) \; \; , \;\; \text{$i=1,2$}, 
$$
defined by the pullback of the universal family $\mathcal X\to H^u$ via $\pr_i\circ \tau$.
From (\ref{eq:lifting}) it follows that there is a commutative diagram 

$$
\xymatrix{
&  \mathbb P^d \times H^u  \ar[dl]  \ar[dr]^{\pr_2}  \\
H^u   &&  \ar[ll]_{\overline \sigma} H^u  \\
&  \ar[ul]^{\pr_1\circ \tau} \Spec (K)  \ar[ur]_{\pr_2\circ \tau}  , 
}
$$
which gives a $\Spec(K)$-isomorphism $\phi: X_1\to X_2$ and a sheaf isomorphism over $\Spec (K)$
$\phi^* \sL_2|_{\Spec (K)}  \longrightarrow  \sL_1|_{\Spec (K)}$. 
As $\mathfrak M^{[N]}_h$ is separated, both extend to isomorphisms over $B$, that is we have 
an isomorphism $\overline\phi: X_1\to X_2$, extending $\phi$, and an isomorphism 
\begin{equation}\label{eq:sheafy}
\phi': (\overline\phi)^* \sL_2  \otimes f_1^* \sB  \longrightarrow  \sL_1  ,
\end{equation}
for some line bundle $\sB$ on $B$.
We note that, similar to the proof of Claim~\ref{claim:divisors}, we 
have natural $B$-isomorphisms 

$$
\rho_i  :  \mathbb P_B \big( (f_i)_*  \sL_i \big)  \longrightarrow   \mathbb P^d \times B .
$$
On the other hand, (\ref{eq:sheafy}) naturally induces the $B$-isomorphism 

$$
\phi'': \mathbb P_B  ((f_1)_* \sL_1 )  \longrightarrow \mathbb P_B ((f_2)_* \sL_2).
$$
The extension $\overline \delta = (\overline \delta_1, \overline \delta_2)$ of $\delta$ can now be 
defined by 

$$
\xymatrix{
b \ar[r]^(0.38){\overline \delta_1} &   \Big( \mathbb P^d\times \{ b \} \ar@/^8mm/[rrr] \ar[r]^{\rho_1^{-1}} & \big[ \mathbb P_B  ((f_1)_* \sL_1 )  \big]_b   
\ar[r]^{\phi''} &  \big[ \mathbb P_B  ((f_2)_* \sL_2 )  \big]_b \ar[r]^{\rho_2}  &    \mathbb P^d \times \{  b\}  \Big)        ,                                                                                                                                     
}
$$
and $\overline \delta_2 = \pr_2\circ \tau$. 

Now, for the second assertion of the claim, using properness of $\overline \psi$, as $G$ is finite over $\mathbb PG$, the composition 

$$
\xymatrix{
\psi: G \times H^u  \ar[r] & \mathbb PG \times H^u \ar[r]^{\overline \psi}  &   H^u \times H^u 
}
$$
is proper. In particular we have properness of the fiber over $\{x\}\times H^u$, which forms the stabilizer of $x$ in $G$. 
Therefore, with $G$ being affine, the stabilizer of $x$ must be finite.  
\end{proof}

\begin{remark}\label{finiteAut}
We note that Claim~\ref{claim:proper} means that for every polarized scheme 
$(Y,L) = \mathfrak M^{[N]}_h(\Spec (\bC))$ 
the group of polarized automorphisms $\mathrm{Aut}(Y,L)$ is finite. 
\end{remark}

\subsection{Connection to the minimal model program} 
As we will see later in Section~\ref{sect:Section4-MainThm},
for a smooth family of projective manifolds with good minimal
models, it is quite useful to have an associated birationally-parametrizing space.

Given two quasi-projective varieties $U$ and $V$, with $V$ being smooth, we say 
$U$ is a \emph{family of good minimal models over $V$}, 
if there is a projective morphism 
$f_{U}: U\to V$ with connected fibers and an integer $N\in \mathbb N$ such that, 
for every $v\in V$, $U_v$ has 
only canonical singularities 
and that the reflexive sheaf 
$\omega_{U/V}^{[N]}$ is invertible and $f_U$-semi-ample
(Definition~\ref{def:relsemi}). 
We sometimes refer to $f_{U}$ as a \emph{relative good minimal model}.

\begin{theorem}[$=$Theorem~\ref{thm:functor}]
\label{functor2}
Let $U''$ be a family of good minimal models over $V''$ via the flat projective morphism $f'': U''\to V''$.
There is a very ample line bundle 
$\sL''$ on $U''$ (not unique) and a polynomial $h$ such that 
$( f'': U'' \to V'', \sL'' ) \in \mfM_h^{[N]}(V'')$.
\end{theorem}

\begin{proof} 
We only need to check the existence of $\sL''$ satisfying \ref{object3}.
Let $\sL$ be a very ample line bundle on $U''$. From the flatness assumption on $f''$ we know 
that $\chi(\sL|_{U''_v})$ is constant for all $v\in V''$.
Therefore, by semicontinuity \cite[Thm.~12.8]{Ha77}, 
for a sufficiently large $a\in \bN$, 
the line bundle $\sL^m$ restricted to each fiber verifies \ref{object3}, for every $m\geq a$. 
That is, for $\sL'':= \sL^a$, we have $H^i( U''_v, ( \sL''_{U''_v} )^b ) =0$, for all $b\geq 1$
and $v\in V$. 
\end{proof}

\begin{remark}
Following the proof of Theorem~\ref{functor2}, we note that if we replace the 
line bundle $\sL''$ by $(\sL'')^m$, for any $m\geq 1$, the conclusions of the theorem 
are still valid.  
\end{remark}

Our next aim is to show that for a suitable choice of a invertible sheaf $\sL$ we can  
ensure that the dimension of subspaces of $M^{[N]}_h$ are closely related to the 
variation of families mapping to them (see Theorem~\ref{VarModuli} below).

\begin{set-up}\label{KawSetup}
Let $f_{U'}: U'\to V$ be a relative good minimal model.
According to \cite[Lem.~7.1]{Kawamata85} there are smooth quasi-projective varieties $\overline V$ and $V''$, 
a surjective morphism $\rho: \overline V \to V''$ and a surjective, generically finite morphism 
$\sigma: \overline V\to V$ with a projective morphisms $f'': U''\to V''$:

\begin{equation}\label{diag:1}
\xymatrix{
U' \ar[d]^{f_{U'}} &&&         &&  U'' \ar[d]^{f''}    \\
V        &&&         \overline{V}  \ar[lll]_{\sigma}^{\text{generically finite}}  \ar[rr]^{\rho} &&   V''   , 
}
\end{equation}
satisfying the following properties. 

\begin{enumerate}
\item \label{Kaw1} Over an open subset $\overline V^\circ\subseteq \overline V$ the morphism $\sigma$ is finite and \'etale. 
\item \label{Kaw2} We have 
$$
 \overline U^\circ : =U'' \times_{V''} \overline{V}^\circ  \cong   U' \times_V  \overline{V}^\circ,
$$
with $\rho':\overline U^\circ \to U''$ and $\sigma': \overline U^\circ \to U'$ being the natural projections.
 
\item \label{Kaw3} For every $t \in \overline V^\circ$ the kernel of $(d_t\rho \circ  d_t\sigma^{-1})$ coincides with the kernel of 
the Kodaira-Spencer map for $f_{U'}: U'\to V$ at $u= \sigma(t)$, where 
$d_t \rho$ and $d_t \sigma$ are the differentials of $\rho$ and $\sigma$.
\end{enumerate}

\begin{theorem-def}[\protect{\cite[Lem.~7.1, Thm.~7.2]{Kawamata85}}]
\label{def:VAR}
For every family of good minimal models $f_{U'}: U'\to V$, the algebraic
closure $K:= \overline{\bC(V'')}$ is the (unique) minimal closed field of definition for $f_{U'}$, that is 
$\Var(f_{U'}) = \dim V''$.
\end{theorem-def}
We note that, as $\Var(\cdot)$ is a birational invariant, for any projective family $f_U:U\to V$ that is birational to a
relative good minimal 
model $U'$ over $V$, we have $\Var(f_U)= \dim V''$. 
\end{set-up}

One can observe that \cite[Lem.~7.1, Thm.~7.2]{Kawamata85} in particular implies that, 
for families of good minimal models, variation is measured at least generically (over the base) 
by the Kodaira--Spencer map. 
Of course this property fails in the absence of the good minimal model assumption
(for example one can construct a smooth projective family of non-minimal varieties of 
general type with zero variation and generically injective Kodaira--Spencer map).
For future reference, we emphasize and slightly extend this point in the following observation. 

\begin{observation}\label{OBSERVE}
We will work in the the situation of Set-up~\ref{KawSetup}.
\begin{enumerate}
\item\label{item:trivial} For every smooth subvariety $T\subseteq \overline{V}^0$, with $\rho(T)$ being a closed point, 
the family $\overline{U}_T^0 \to T$ is trivial. In particular, if $\Var(f_{U'}) =0$, then $f_{U'}$ is generically (over V) isotrivial.
\item\label{item:trivialPol} For every $T\subseteq \overline{V}^0$ as in Item~\ref{item:trivial} 
and line bundle $\sL''$ on $U''$, the polarized family $\big(  \overline{U}^0_T\to T, (\rho')^* \sL''  \big)$ is trivial. 
\end{enumerate}
To see this, we may assume that $\overline V^0 = \overline V$. Set $v'':= \rho(T) \in V''$. 
By the assumption we have 
$$
\overline U_T \cong T \times_{\bC} F 
$$
where $F:= U''_{v''}$ (which shows Item~\ref{item:trivial}). Thus, over $T$, $\rho'$ coincides with 
the natural projection $\pr_2: T\times_{\bC} F\to F$. Clearly, $(T \times_{\bC} F, \pr_2^*\sL''_{v''})$ is trivial. 
\end{observation}

\begin{theorem}\label{VarModuli}
In the setting of Set-up~\ref{KawSetup}, over an open subset $V_{\eta}$ of $V$
there is a line bundle $\sL$ (not unique) such that $(f_{U'}: U'_{\eta} \to V_{\eta} , \sL ) \in \mfM_h^{[N]}(V_{\eta})$, 
with the induced morphism $\mu_{V_{\eta}}: V_{\eta} \to M_h^{[N]}$ verifying the 
equality  
\begin{equation}\label{eq:VarEq}
\Var(f_{U'}) = \dim \big( \Im(\mu_{V_{\eta}}) \big).
\end{equation}
In particular, any relative good minimal model $f_{U'}: U'\to V$ of any smooth family $f_U: U\to V$
of projective varieties with good minimal model gives rise to a morphism of this form.
\end{theorem}

\begin{proof}
We start by considering Diagram~\ref{diag:1}. 
In Set-up~\ref{KawSetup} we may assume that $\overline V= \overline V^\circ$. 
Denote $U'\times_V \overline V$ by $\overline U$ and set $\overline f: \overline U \to \overline V$ to be the pullback family.
Using Item~\ref{Kaw2}, generically, the morphism $f''$ is a family of good minimal models (see also 
the global generation case in the proof of Proposition~\ref{prop:good}), that is  
after replacing $V$ by an open subset $V_{\eta}$ we can assume that $f''$ is a relative good minimal model and flat. 
Let $\sL''$ 
be a choice of line bundle as in the proof of Theorem~\ref{functor2} so that 
$(f'': U''\to V'', \sL'') \in \mfM_h^{[N]}(V'')$.
We may assume that $\sigma$ is Galois, noting that if $\sigma$ is not Galois, we can replace 
it by its Galois closure and replace $\rho$ by the naturally induced map.
Define $G:= \Gal(\overline V/V_{\eta})$.

Now, we define $\sL''_{\overline V} := (\rho')^*\sL''$ and consider the $G$-sheaf 
$\bigotimes_{g\in G} g^*\sL''_{\overline V} \cong(\sL''_{\overline V})^{|G|}$ 
(see for example \cite[Def.~4.2.5]{MR2665168} for the definition). 
As $\sigma'$ is \'etale, the stabilizer of any point $\overline u\in \overline U$ 
is trivial (and thus so is its action on the fibers). Consequently, the above $G$-sheaf 
descends~\cite[Thm.~4.2.15]{MR2665168}. 
That is, there is a line bundle $\sL$ on $U'_{V_{\eta}}$ 
such that 
$$
(\sigma')^* \sL \cong \bigotimes_{g\in G} g^* \sL''_{\overline V} .
$$
Therefore, we have 
$$
(f': U'_{V_{\eta}} \to V_{\eta} , \sL) \in \mfM_h^{[N]}(V_{\eta}). 
$$
After replacing $\sL''$ by $(\sL'')^{|G|}$, so that $(\rho')^*\sL'' =(\sL''_{\overline V})^{|G|}$, 
we can ensure that the Hilbert polynomial of $\overline f$ with respect to $(\sL''_{\overline V})^{|G|}$ 
is equal to the one for $f''$ with respect to $\sL''$. 
Let $\mu_{V_{\eta}}: V_{\eta}\to M_h^{[N]}$ denote the induced 
moduli map. 

Our aim is now to establish the equality (\ref{eq:VarEq}). 
To this end, set $W$ to be the image of 
$V_{\eta}$ under $\mu_{V_{\eta}}$.

\begin{claim}\label{claim:dimension}
$\dim W  \geq \dim V''$.
\end{claim}

\noindent
\emph{Proof of Claim~\ref{claim:dimension}.}
Assume that instead $\dim W < \dim V''$.
Let $T\subseteq V_{\eta}$ be a subscheme whose pre-image under $\sigma$ is
generically finite over $V''$. This implies that $\dim T = \dim V''$
and that the variation of the induced family over $T$ defined by pullback
of $f''$ is maximal (see Item~\ref{Kaw3}). 
Now, by comparing the dimensions, we see that 
$$
\dim\big( \mu_{V_{\eta}}(T)   \big)  <   \dim T.
$$
But this contradicts the fact that the induced family over $T$ has 
maximal variation.
This can be seen as a consequence of Koll\'ar's result ~\cite[Cor.~2.9]{Kollar87}
for families of varieties with non-negative Kodaira dimension. \qed 

Now, let $Z\subseteq V_{\eta}$ be a subscheme that is generically finite and dominant over $W$.
By construction, the induced moduli map $\mu_{\overline V}: \overline V\to M_h^{[N]}$ 
associated to $(\overline f: \overline U\to \overline V, {\sL''_{\overline V}}^{|G|})$
factors through $\mu_{V_{\eta}}$. 
Therefore, $\sigma^{-1}(Z)$ is also generically finite over $W$. 

\begin{claim}\label{SecondClaim}
$\sigma^{-1}(Z)$ is generically finite over $V''$ and thus $\dim W\leq \dim V''$. 
\end{claim}

\noindent
\emph{Proof of Claim~\ref{SecondClaim}.} If $\rho|_{\sigma^{-1}Z}: \sigma^{-1}Z\to V''$ is not generically finite, then 
for each irreducible (positive dimensional) general fiber $T\subseteq \sigma^{-1}(Z)$ mapping to a smooth closed point $v''$
Observation~\ref{OBSERVE} (Item~\ref{item:trivialPol}) applies. 
Therefore, as the above choice of the polarization for the family defined by $\overline f$ is pullback 
of the one fixed for $f''$ via $\rho'$,
the family $(\overline U_T, (\sL''_{\overline V})^{|G|} = (\rho')^*\sL'')$
is locally trivial as polarized schemes. 
Thus, by the construction of $\mu_{\overline V}$, 
the general fiber of $\sigma^{-1}(Z) \to V''$ must be contracted by $\mu_{\overline V}$, contradicting the 
generic finiteness of $\mu_{\overline V}|_{\sigma^{-1}(Z)}$. \qed


The first half of the theorem now follows from Claims~\ref{claim:dimension} and~\ref{SecondClaim}.

To see that every smooth projective family $f_U:U\to V$ of varieties admitting a good minimal model 
leads to a moduli morphism $\mu_{V_{\eta}}: V_{\eta}\to M_h^{[N]}$ as above, for some $N\in \bN$, 
using the first half of the theorem, 
it suffices to know that $f$ has a relative good 
minimal model $f_{U'}: U'\to V$. But this is guaranteed, for example by \cite[1.2, 1.4]{HMX18}. 
\end{proof}

\begin{notation}[replacing $M^{[N]}_h$ by an \'etale covering]
\label{not:cover}
Let $M \to M^{[N]}_h$ be an \'etale covering, with $M$ being a finite type scheme, cf. \cite[Chapt.~2]{Knu71} 
(see also \cite[pp.~279--280]{Viehweg95}). 
Set $\mu'_{V_{\eta}}: V_{\eta}\to M$ to be the finite type morphism of schemes representing 
$\mu_{V_{\eta}}$ in this \'etale covering. 
Let $M^0\subseteq M$ be an affine subscheme containing the generic point of 
$\Im(\mu'_{V_{\eta}})$. After replacing $V_{\eta}$ by $V^0_{\eta}:= (\mu'_{V_{\eta}})^{-1}M^0$
we thus have a finite type morphism 

$$
\mu'_{V_{\eta}}: V^0_{\eta} \longrightarrow  M^0 
$$
of quasi-projective schemes. 
By abuse of notation, from now on we will denote $V^0_{\eta}$, $M^0$ by $V_{\eta}$ and $M^{[N]}_h$, 
respectively, that is $M^{[N]}_h$ is quasi-projective and $\mu_{V_{\eta}}: V_{\eta}\to M^{[N]}_h$ 
is the induced moduli map. 
\end{notation}

\begin{corollary}\label{cor:FinalModuli}
Let $f:X\to B$ be a smooth comapctification of a smooth projective family $f_U: U\to V$, whose 
fibres admit good minimal models. Then, depending on a choice of a relative good minimal model 
for $f_U$ there is a polarization $\sL$ as in Theorem~\ref{VarModuli} and, following the notation introduced in
Notation~\ref{not:cover}, 
there is a rational 
moduli map $\mu_{V_{\eta}}: B \dashrightarrow \overline M_h^{[N]}$, where 
$\overline M_h^{[N]}$ is a compactification of $M_h^{[N]}$ by a projective scheme. 
Moreover, we have $\dim (\Im \mu_{V_{\eta}}) = \Var(f_U)$. 
\end{corollary}

\section{Base spaces of families of manifolds with good minimal models}
\label{sect:Section4-MainThm}

To prove Theorems~\ref{thm:sheaves} and~\ref{thm:iso}  we will use the moduli functor in 
Theorem~\ref{VarModuli} to construct a new family $f_Z: X_Z \to Z$ out of the initial $f: X\to B$ over which 
the variation is maximal (Proposition~\ref{prop:maximal} below). 
Serving as a key component of the proof of Theorem~\ref{thm:sheaves}, the subsystems of canonical extensions of 
VHS 
in Section~\ref{sect:Section2-Subsheaves} will then be constructed for $f_Z$
and various families arising from it (see \cite[Lem.~2.8]{VZ02}).

\begin{notation}
For a flat morphism $f: X\to Y$ of regular schemes, by $X^{(r)}$ we denote a strong desingularization of
the $r$-th fiber product over $Y$
$$
X^r: = \underbrace{X \times_Y X \times_Y \ldots \times_Y X}_{\text{$r$ times}}.
$$
\end{notation}


Noting that as $M_h^{[N]}$ is taken to be quasi-projective (Theorem~\ref{thm:functor} and Notation~\ref{not:cover}), 
as in Corollary \ref{cor:FinalModuli} by $\overline M_h^{[N]}$ we denote its projective compactification. 
Since in what follows all our maps to $\overline{M}_h^{[N]}$ originate from reduced schemes, 
with no loss of generality 
we will assume that $\overline{M}_h^{[N]}$ is already reduced.

\begin{proposition}\label{prop:maximal}
In the setting of Corollary~\ref{cor:FinalModuli}, assume that 
$\Var(f_U)\neq 0$, $\Var(f_U)\neq \dim (V)$ (see Definition~\ref{def:VAR}) and set $n=\dim(X/B)$.
After replacing $B$ by a birational model, let $\overline \mu_{V_{\eta}}: B\to \overline M_h^{[N]}$ 
be a desingularization of $\mu_{V_{\eta}}$. 
Then, 
there are smooth projective  
varieties $Z^+$ and $Z$, a morphism $\gamma: Z^+ \to B$ and, after removing a
subscheme of $Z^+$ of $\codim_{Z^+}\geq 2$, a morphism $g: Z^+\to Z$ that fit 
into the commutative diagram

\begin{equation}\label{KeyDiagram}
\xymatrix{
X \ar[d]_f    &  X' \ar[l]_{\gamma'} \ar[drr]^{f'} &&  X_Z^+  \ar[d]^{f_Z^+}  \ar@{-->}[ll]_{\pi}  \ar[rr]     &&  X_Z \ar[d]^{f_Z} \\
B    &&&  Z^+  \ar[lll]_{\gamma}  \ar[rr]^g &&    Z,
}
\end{equation}
verifying the following properties.

\begin{enumerate}
\item \label{item:1} We have $\dim(B) = \dim(Z^+) > \dim(Z)$, where $\dim Z = \Var(f_U)$.
Moreover, there is a morphism
$\overline{\mu}_{V_{\eta}}: B \to W$ with positive relative dimension, connected fibers and a generically finite map
$\mu_Z : Z  \to W$ such that the diagram

$$
\xymatrix{
B   \ar[drrr]_{\overline{\mu}_{V_{\eta}}} & Z^+  \ar[l]_{\gamma}  \ar[rr]^g  &&  Z    \ar[d]^{\mu_Z}  \\
      &&& W
}
$$
commutes. 

\item \label{item:2} The morphism 
$g: Z^+ \to Z$ is a codimension one flattening of a proper morphism (see Definition~\ref{def:flatten}). 

\item \label{item:3} The two schemes $X_Z^+$ and $X_Z$ are regular and quasi-projective. The 
two morphisms $f_Z$ and $f_Z^+$ are projective with connected fibers.
Moreover, $f_Z$ is semistable. 

\item \label{item:4} With $X'$ being a strong desingularization of $X \times_B  Z^+$, there is a birational map 
$\pi: X^+_Z \dashrightarrow X'$ over $Z^+$. The morphism 
$f': X' \to Z^+$ is the naturally induced map. 

\item \label{item:5} For any $r\in \bN$ there is an induced diagram 
involving similarly defined morphisms $f^{(r)}: X^{(r)} \to B$, 
$f_Z^{(r)}: X^{(r)}_Z \to Z$, ${f_Z^+}^{(r)} : {X_Z^+}^{(r)}\to Z^+$ 
and ${f'}^{(r)} : {X'}^{(r)} \to Z^+$ 
commuting with the ones in Diagram~\ref{KeyDiagram}.

\item \label{item:6} For any sufficiently large and divisible $m$, the line bundle 
defined by the reflexive hull of
$\det (f_Z)_* \omega^{m}_{X_Z/Z}$ is big, 
implying that for a sufficiently large $N\in \bN$ there is an ample line bundle 
$\sA_Z:= \big(\det (f_Z)_* \omega^{m}_{X_Z/Z}\big)^{[N]} (- D_Z)$, 
for some effective divisor $D_Z\geq D_{f_Z}$ on Z.
Moreover, for sufficiently large integers $m$, 
there is $r\in \bN$ such that $H^0 ( X^{(r)}_Z , \sM^m ) \neq 0$, where 
\begin{equation}\label{SmallSheaf}
\sM: = \Omega^{rn}_{X_Z^{(r)}/Z} ( \log \Delta_{f_Z^{(r)}} ) \otimes ( f_Z^{(r)} )^* \sA_Z^{-1} .
\end{equation}


 \end{enumerate}
\end{proposition}
\begin{proof}
Let $W$ be the image of $\overline{\mu}_{V_\eta}$. 
Using Stein factorization we replace $W$ by a finite
covering so that $\overline \mu_{V_{\eta}}$ has connected fibers. 
Take $Z\subset B$ to
be a sufficiently general, smooth and complete-intersection subvariety 
such that $\mu_Z: = \overline{\mu}_{V_{\eta}}|_{Z}: Z\to W$ is generically finite.
By Corollary~\ref{cor:FinalModuli} we have $\dim Z =\Var(f_U)$.
 
Define $Z^+$ to be a desingularization of the normalization of 
$B\times_W Z$. Let $\gamma: Z^+ \to B$ be the resulting naturally defined map.

\smallskip

For Item~\ref{item:2}, 
let $\wtilde g: \wtilde{Z^+} \to Z'$ be a flattening of $g$ so that, after
removing a subscheme of $Z^+$ of $\codim_{Z^+}\geq 2$, the induced map
$g: Z^+ \to Z'$ is a codimension one flattening. We now replace $Z$ by $Z'$.

\smallskip

As for Item~\ref{item:3}, take $X_Z$ to be a strong desingularization of the pullback of $f: X\to B$ via 
the morphism $Z\to B$. 
Let $\widehat Z\to Z$ be a cyclic, flat morphism associated to a semistable reduction 
$\widehat f_Z: Z_{\widehat Z}\to \widehat Z$ in codimension one for $f_Z$.
Again, after removing a subset of $Z$ of $\codim_{Z}\geq 2$ (and therefore of $Z^+$ as $g$ is flat), 
we replace $f_Z$ by $\widehat f_Z$.
$X_{Z}^+$ is a 
desingularization of $X_Z \times_Z Z^+$. 

\smallskip

For Item~\ref{item:4}, let $V_{\eta}$ be as in Theorem~\ref{VarModuli}. Let $U'$ be a good minimal model 
over $V_{\eta}$ and set $Z_{\eta}$  to denote the restriction of $Z$ to $V_{\eta}$. 
Define $Z_{\eta}^+:= Z_{\eta} \times_{W} V_{\eta}$. As before 
$U'_{Z_{\eta}}$ and $U'_{Z_{\eta}^+}$ denote the pullback of $U'$ 
via $Z_{\eta}\to V_{\eta}$ and $Z_{\eta}^+\to V_{\eta}$, respectively. 
Next, define $(U'_{Z_{\eta}})^+$ to be the pullback of $U'_{Z_{\eta}}\to Z_{\eta}$
through $Z^+_{\eta}\to Z_{\eta}$.
Summarizing this construction we have:

$$
\xymatrix{
&&                   (U'_{Z_{\eta}})^+ \ar[rr]  \ar[d]  &&    U'_{Z_{\eta}}  \ar[d]  \\
U'_{Z_{\eta}^+}  \ar[rr]  \ar[d] &&    Z^+_{\eta}   \ar[rr]  \ar[d]  &&   Z_{\eta}  \ar[d]  \\
U'  \ar[rr]   &&    V_{\eta}  \ar[rr]            &&   W    .
}
$$
\begin{claim}\label{claim:ISOM}
Up to a finite covering of $Z^+_{\eta}$, $(U'_{Z_{\eta}})^+$ is isomorphic to $U'_{Z_{\eta}^+}$ over $Z^+_{\eta}$.
\end{claim}

\noindent 
\emph{Proof of Claim~\ref{claim:ISOM}.} 
This follows from the above construction, finiteness of the polarized 
automorphism groups in Remark.~\ref{finiteAut} and the following fact, 
which is a consequence of representability of the isomorphism scheme 
for polarized projective varieties (cf.~\cite[Sect.~7]{KK10} for the canonically polarized case). 

\noindent
\emph{Fact.}
Assume that $X_i$ and $Y$ are quasi-projective varieties, $i=1,2$.
Let $f_i: (X_i \to Y , \sL_i )$ be two polarized flat projective families of varieties 
such that for every $y \in Y$ we have  $| \mathrm{Aut} (X_i, \sL_i)_y  |< \infty$
and $(X_1, \sL_1)_y \cong (X_2, \sL_2)_y$. 
Then, there is a finite surjective morphism $\sigma : Y'\to Y$ such that 
$(X_1)_{Y'} \cong (X_2)_{Y'}$, extending the fiber-wise isomorphism. 
\qed

Therefore, without loss of generality we may replace $Z^+$ by the finite covering defined in Claim~\ref{claim:ISOM}.
Consequently, we find a birational map $\pi: X^+ \dashrightarrow   X'$ over $Z^+$ as in 
Item~\ref{item:4}.
%
%

\smallskip

Item~\ref{item:5} can be easily checked.

Item~\ref{item:6} is a deep result of Kawamata~\cite[Thm.~1.1]{Kawamata85}
for smooth families of projective varieties with good minimal 
models, assuming that variation is maximal. 

Moreover, given $r_m:=\mathrm{rank}((f_Z)_* \omega^m_{X_Z/Z})$, we recall the natural 
inclusion 
\begin{equation}\label{eq:FF}
\big( \det(f_Z)_* \omega^m_{X_Z/Z}\big)^m  \subseteq \bigotimes^{mr_m} (f_Z)_* \omega^m_{X_Z/Z}  ,
\end{equation}
where, using the semistability of $f_Z$, the right-hand side is isomorphic to 
$(f_Z^{(r)})_* \omega^m_{X_Z^{(r)}/Z}$, with $r:=mr_m$, cf.~\cite[Sect.~3]{Viehweg83}.
This implies that 
$$
h^0\Big( \omega_{X_Z^{(r)}/Z} \otimes (f_Z^{(r)})^* (\det (f_Z)_* \omega^m_{X_Z/Z})^{-1} \Big)^m\neq 0.
$$
(Note that following the setting of the proposition we may ignore closed subset of $\codim_{Z}\geq  2$.)

Now, after replacing the power $m$ on the left-hand side of (\ref{eq:FF}) by $Nm$, and using 
the definition of $\sA_Z$ we find that 
\begin{equation}\label{eq:FFF}
h^0\Big(  \omega^m_{X_Z^{(r)}/Z}  \otimes (f_Z^{(r)})^* \big(   \sA_Z(D_Z)  \big)^{-m}   \Big)  \neq 0  ,
\end{equation}
where $r$ is now set as $Nmr_m$. 
As $D_{f_Z}= D_{f_Z^{(r)}}$ and $D_Z\geq D_{f_Z}$, we have 
$$
\omega^m_{X_Z^{(r)} /Z}  \otimes (f_Z^{(r)})^* (\sA_Z(D_Z))^{-m}  
\subseteq  \Big( \Omega^{rn}_{X_Z^{(r)}} \big( \log \Delta_{f_Z^{(r)}}  \big)  \Big)^{\otimes m} \otimes (f_Z^{(r)})^* \sA_Z^{-m}, 
$$
which together with (\ref{eq:FFF}) implies (\ref{SmallSheaf}).

 
\end{proof}

\subsection{Proof of Theorem~\ref{thm:sheaves}}\label{SSS}
Let $f: X\to B$ be a smooth compactification of $f_U$
so that, consistent with the rest of this paper, $D$ in the setting of the theorem will be replaced by the notation $D_f$.
When $\Var(f_U) = \dim B$, the theorem is due to~\cite{Vie-Zuo03b}, in the canonically polarized case,
and~\cite{PS15} in general (see also~\cite{Taj18}). So assume that $\Var(f_U) \neq \dim B$.

By Proposition~\ref{prop:maximal} we know that $f: X\to B$ fits inside 
the diagram (\ref{KeyDiagram}). For now let us identify $f$ with its base change. 
After replacing $X$ by $X^{(r)}$, for sufficiently large $r$ (Item~\ref{item:5}), let 
$\sM$ be the line bundle on $X_Z$ defined 
in~\ref{item:6}, namely 

\begin{equation}\label{eq:LB}
\sM  =  \Omega^n_{X_Z/Z} ( \log \Delta_{f_Z} )  \otimes f_Z^*\sA_Z^{-1}. 
\end{equation}
By Item~\ref{item:6} we know that $H^0( X_Z, \sM^m ) \neq 0$. 
For the moment we will assume that $\gamma$ is finite (and therefore flat). 
As such, and using the birational map $\pi$, the constructions and conclusions 
of Lemmas~\ref{PrelimLem},~\ref{lem:copy} and Proposition~\ref{Prop:GoodSS} 
are valid. As $\sG''_0= g^*\sA_Z$ (see Set-up~\ref{SetUpLem}), 
by using the map (\ref{eq:KS}), for some $k\in \mathbb N$, 
we have an injective morphism 
$$
 \sL_{Z^+}:= g^*( \underbrace{\sA_Z \otimes \sB_m}_{:= \sL_Z} )  \cong g^*\sA_Z  \otimes (\det \sN^+_m)^{-1}  
                \hooklongrightarrow  \big(   \gamma^*\Omega^1_B(\log D_f)   \big)^{\otimes k}.
$$
As $\sA_Z$ is big in $Z$ and $\mu_Z$ is generically finite, we have $\kappa( Z^+ , g^* \sL_Z ) \geq \Var(f) $. 
Let us denote the saturation of the image of 
\begin{equation}\label{eq:KeyInclusion}
 \sL_{Z^+} = g^* \sL_Z   \hooklongrightarrow  \gamma^* \big(   \Omega_B^1 (\log D_f)  \big)^{\otimes k}.
\end{equation}
by $\overline \sL_{Z^+}$.
After deleting appropriate subscheme of $B$ of $\codim_B\geq 2$, using its Galois closure, 
we may also assume that $\gamma: Z^+ \to B$ is Galois.
Set $G:= \Gal(Z^+/B)$.
It follows that the $G$-sheaf $\bigotimes_{a\in G} a^*\sL_{Z^+}$ descends~\cite[Thm.~4.2.15]{MR2665168}, that is

\begin{equation}\label{eq:descent}
 \bigotimes_{a\in G} a^*\overline\sL_{Z^+} \cong \gamma^* \sL,
  \end{equation}
for some line bundle $\sL$ on $B$. Therefore, 
the two line bundles in (\ref{eq:descent}) have the same
Kodaira dimension, cf.~\cite[Lem.~10.3]{Iitaka82} and we have
$ \kappa(B, \sL) \geq \kappa(Z^+, \sL_{Z^+}) \geq \Var(f_U)$,
as required.

Since, after removing a $\codim_B\geq 2$-subscheme of $B$, the morphism $\gamma$
is finite, and as our ultimate goal is birational, following the above argument,
we may assume with no loss of generality that $\gamma$ is indeed finite. 
This is not difficult to check (using Stein factorization) and we leave the details to the reader.

Finally, we note that the above argument shows that we may assume with 
no loss of generality that $f:X\to B$ is identified with the required finite base change  
in Proposition~\ref{prop:maximal}. This finishes the proof of Theorem~\ref{thm:sheaves}.

\medskip

\subsection{Proof of Theorem~\ref{thm:iso}}
To prove Theorem~\ref{thm:iso} we need a refinement of the statement 
of Theorem~\ref{thm:sheaves} in the sense of Theorem~\ref{thm:refine} . 
As one would expect, the proof of 
this refinement is inextricably intertwined with that of Theorem~\ref{thm:sheaves} itself.
We note that
in the canonically polarized case this refinement is due to 
Jabbusch-Kebekus~\cite{MR2976311}. 
The notion of \emph{orbifolds}, as developed by Campana, is the key ingredient 
for realizing this improvement of Theorem~\ref{thm:sheaves}. 
We refer to the original paper of Campana \cite{Cam04} 
for the basic definitions and further background. 
For reader's convenience, a brief summary of all required notions 
in this theory
has been included in the appendix.

\begin{set-up}\label{finalsetup}
We will be working in the setting of Proposition~\ref{prop:maximal}.
The line bundles $\sL_{Z^+}$ and $\sL_Z$ are the ones defined in (\ref{eq:KeyInclusion}) and (\ref{eq:descent}).
To lighten up the notations we will use $\overline \mu$ to denote $\overline \mu_{V_{\eta}}$. 
Let $D_f = D_f^v + D_f^h$ be the decomposition of 
$D$ into vertical and horizontal components with respect to $\overline \mu$. 
Let $W^0\subseteq W$ be the maximal open subset over which $\overline \mu: (B, D_f)\to W$
is neat.
By construction, over $W^0$ there is a natural map 
\begin{equation}\label{OrbiPB}
\overline \mu^* \big(  \Omega^1_{W^0} (\log \Delta_{W^0})^{\otimes_{\mathcal C} k}  \big)  
 \longrightarrow  \big( \Omega^1_B(\log D^v_f)  \big)^{\otimes k}, 
\end{equation}
(cf.~\cite[Sect.~5.B]{MR2860268}).
Set $\sB$ to be the saturation of the image of (\ref{OrbiPB}).
\end{set-up}

\begin{proposition}\label{finalProp}
Assume that $W^0=W$. Let $\sL$ be the line bundle in (\ref{eq:descent}). 
There is an injection $\sL \hookrightarrow \sB$.
\end{proposition}

\begin{proof}
Let $\sQ$ be the torsion free cokernel of $\sB \to \Omega^1_B(\log D_f^v)^{\otimes k}$.
Since $\gamma$ is flat, we have the short exact sequence 
$$
0 \longrightarrow  \gamma^*\sB \longrightarrow  \gamma^*\big( \Omega^1_B(\log D_f^v) \big)^{\otimes k} \longrightarrow \gamma^*\sQ \longrightarrow 0, 
$$
with $\gamma^*\sQ$ being torsion free in codimension one. 
\begin{claim}\label{claim:needinject}
Let $\sL_{Z^+}$ be the line bundle on $Z^+$ defined in (\ref{eq:KeyInclusion}). After removing 
a subset of $Z^+$ of $\codim_{Z^+}\geq 2$ if necessary,
the injection 
$$
i : \sL_{Z^+} \hooklongrightarrow  \gamma^*\big( \Omega^1_B(\log D_f) \big)^{\otimes k}
$$
factors through $\gamma^*\sB \subseteq  \gamma^*\big(  \Omega^1_B(\log D_f^v) \big)^{\otimes k}$.
\end{claim}
\noindent
\emph{Proof of Claim~\ref{claim:needinject}.}
First, we observe that by Lemma~\ref{lem:KS} 
we have $i: \sL_{Z^+} \hooklongrightarrow 
\gamma^*\big(  \Omega^1_B(\log D_f^v) \big)^{\otimes k} \subseteq \gamma^*\big( \Omega^1_B(\log D_f) \big)^{\otimes k}$.
Furthermore, again by Lemma~\ref{lem:KS}, over an open subset $Z^+_0 \subseteq Z^+$ (given by $Z^+\backslash D_{f^+}$) 
we have $i: \sL_{Z^+}|_{Z^+_0} \hookrightarrow (g^*\Omega^1_Z)|_{Z^+_0}$.
Using the commutativity of the diagram 
$$
\xymatrix{
 Z^+ \ar[rr]^{\gamma}  \ar[d]_{g}  &&  B  \ar[d]^{\overline \mu} \\
Z   \ar[rr]^{\mu_Z}     &&     W  ,
}
$$
this in particular implies that over an open subset of $W$ the line bundle $\sL_{Z^+}$
injects into $g^* \mu_Z^* \big(  \Omega^1_W(\log \Delta_W)^{\otimes_{\mathcal C} k}  \big)
= \gamma^* \overline\mu^* \big( \Omega^1_W(\log \Delta_W)^{\otimes_{\mathcal C} k}  \big)$. 
This means that at least over an open subset of $Z^+$ the factorization 
in Claim~\ref{claim:needinject} holds. In other words, the naturally induced map 
$$
\sL_{Z^+} \hookrightarrow  \gamma^* \sQ
$$
has a nontrivial kernel. 
As $\gamma^*\sQ$ is torsion free in codimension one, it follows 
that this map is zero in codimension one, implying 
the desired injection in Claim~\ref{claim:needinject}.  \qed

Now, we may assume with no loss of generality that the inclusion
$\sL_{Z^+} \subseteq \gamma^*\big(\Omega^1_B(\log D)\big)^{\otimes k}$ 
is saturated 
and that $\gamma$ is Galois. 
By (\ref{eq:descent}) we have $\sL_{Z^+}^{|G|} \cong \gamma^*\sL$. 
We may also assume that $|G|=1$ (as we may replace $k$ by $k |G|$). 
Now, by applying $\gamma_*(\cdot)^G$ to the 
injection $\gamma^*\sL\cong \sL_{Z^+} \hookrightarrow \gamma^*\sB$ in Claim~\ref{claim:needinject}
we find the injection $\sL \hookrightarrow \sB$ in codimension one, which as $\sB$ is reflexive, 
extends to an injection over $B$.
\end{proof}

\begin{theorem}\label{thm:refine}
In the situation of Setup~\ref{finalsetup}, let $\wtilde \mu: (\wtilde B, \wtilde D)\to \wtilde W$ 
be a neat model for $\overline \mu: (B, D_f) \to W$, 
via birational morphisms $\alpha$ and $\pi$ (see Definition~\ref{def:neatmodel}),
and with the orbifold base $(\wtilde W, \Delta_{\wtilde W})$ (Definition~\ref{orbibase}).
Let $\wtilde \sB$ be the saturation of the image of 
$$
(\wtilde \mu)^* \big( \Omega^1_{\wtilde W}  (\log \Delta_{\wtilde W})^{\otimes_{\mathcal C}k}  \big)  
\longrightarrow  \Omega^1_B \big(  \log (\wtilde D)^v \big)^{\otimes k}  ,
$$
where $(\wtilde D)^v$ denotes the vertical component of $\wtilde D$ (with respect to $\wtilde \mu$). 
Then, there is a line bundle $\wtilde \sL$ on $\wtilde B$, with $\kappa(\wtilde \sL) = \kappa(\sL)$, and equipped with 
an injection $\wtilde \sL \hookrightarrow \wtilde \sB$.
\end{theorem}

\begin{proof}
Let $\wtilde Z$ be a desingularization of the main component of
$\wtilde B \times_B Z^+$, with the naturally induced maps 
$\wtilde \gamma: \wtilde Z\to \wtilde B$ and $\wtilde\pi: \wtilde Z\to Z^+$. 
Set $\sL_{\wtilde Z}:= (\wtilde \pi)^*\sL_{Z^+}$.  
The proof is now the same as that of Proposition~\ref{finalProp} after replacing 
$\overline\mu: (B,D_f)\to W$, $\gamma:Z^+\to B$ and $\sL_{Z^+}$, 
by $\wtilde \mu: (\wtilde B, \wtilde D)\to \wtilde W$, $\wtilde \gamma: \wtilde Z\to \wtilde B$ 
and $\sL_{\wtilde Z}$, respectively. 
\end{proof}

\subsubsection{Generic descent to coarse space as an orbifold base; conclusion of the proof of Theorem~\ref{thm:iso}}
With Theorem~\ref{thm:refine} at hand, noting that $\kappa(\sL) \geq \dim(W)$, 
the proof of Theorem~\ref{thm:iso} is now identical to \cite{Taji16}, for which~\cite{CP14} or~\cite{CP16} provides a vital ingredient 
(see also Claudon's  Bourbaki exposition~\cite{Claudon15}).

Aiming for a contradiction, we assume that $f_U$ is not isotrivial. 
Thanks to the already established results in the maximal variation case (\cite{Vie-Zuo03b},~\cite{CP14}  and~\cite{PS15})
we know $\Var(f_U) \neq \dim(V)$. In particular the constructions of Proposition~\ref{prop:maximal} 
and those of Subsection~\ref{SSS} apply. 
Furthermore, as specialness is a birational invariant for log-smooth pairs, 
we may replace $(B,D) \to W$ (with $D=D_f$) by its neat model
$(\wtilde B, \wtilde D)\to \wtilde W$ as in Theorem~\ref{thm:refine}. 
At this point~\cite[Cor.~5.8]{MR2860268} applies, 
that is there is a line bundle $\sL_{\wtilde W}$ 
in $\big(\Omega^1_{\wtilde W}(\log \Delta_{\wtilde W})\big)^{\otimes_{\mathcal C}N}$, for some $N\in \bN$, 
with $\kappa_{\mathcal C} (\wtilde W,  \sL_{\wtilde W})= \dim \wtilde W$ (see~Definition~\ref{ckappa} 
for the definition of $\kappa_{\mathcal C}$).
It then follows that $\kappa( \wtilde W, \Delta_{\wtilde W} ) = \dim \wtilde W$~\cite[Thm.~5.2]{Taji16}, 
which in turn implies that $(\wtilde B, \wtilde D)$ is not special, and thus 
neither is $(B, D)$, contradicting our initial assumption.

\section{Appendix: background on orbifolds}
In this appendix for the convenience of the reader we review some basic elements 
of Campana's theory of orbifolds \cite{Cam04}. 
Most of the definitions in this appendix have been taken from~\cite{Taji16}.

\begin{definition}
A pair, or an orbifold pair, consists of a variety $Y$ and a $\mathbb Q$-Weil divisor 
$D= \sum d_i D_i$, where $d_i \in [0,1]\cap \mathbb Q$ and each $D_i$ is prime. 
We say $(Y,D)$ is snc, if $Y$ is smooth and $D$ has a simple normal crossing support.   
\end{definition}

Given a quasi-projective morphism $h: (Y,D) \to W$ with connected fibers and $\mathbb Q$-factorial $W$, there is 
a pair $(W, \Delta_W)$, referred to as the \emph{orbifold base} or $\mathcal C$-base associated to $h$ and $D$, see for 
example~\cite[Def.~5.3]{MR2860268} or Definition~\ref{orbibase}.

\begin{definition}
We say $h:(Y,D) \to W$ as above is neat, if 
\begin{enumerate}
\item $(Y,D)$ and the orbifold base $(W, \Delta_W)$ are snc, and
\item every $h$-exceptional divisor $P\subset Y$ ($\codim_W(h(P)) \geq 2$) is contained in $\Supp(D)$
as a reduced divisor. 
\end{enumerate}
\end{definition}

\begin{definition}[\protect{cf.~\cite[Def.~4.1]{Taji16}}]
\label{def:neatmodel}
Given any pair $(Y,D)$ and a morphism $h: Y \to W$, a neat 
morphism $\wtilde h: (\wtilde Y, D_{\wtilde Y}) \to \wtilde W$
is called a \emph{neat model} for $h:(Y,D)\to W$, if
\begin{enumerate}
\item $\wtilde h: \wtilde Y \to \wtilde W$ is birationally equivalent to $h: Y\to W$, i.e.,
we have a commutative diagram 
$$
\xymatrix{
\wtilde Y \ar[rr]^{\pi}  \ar[d]_{\wtilde h} &&  Y \ar[d]^h  \\
\wtilde W  \ar[rr]^{\alpha}    &&   W  , 
}
$$
where $\alpha, \pi$ are birational morphisms, and  
\item $D_{\wtilde Y}$ is the sum of the total $\pi$-transform 
of $D$ and all $\wtilde h$-exceptional divisors in $\wtilde Y$ (as reduced divisors). 
\end{enumerate}
\end{definition}

\begin{remark}
For any $(Y,D)$ and $h:Y\to W$ as in Definition~\ref{def:neatmodel} a neat model can be constructed (not uniquely), 
cf.~\cite[Sect.~10]{MR2860268} or \cite[Prop.~4.2]{Taji16}. 
The construction is originally due to Campana. 
\end{remark}

\begin{notation}
Given an snc pair $(Y, D)$, for each $i \in \mathbb N$, by 
$\Omega^1_Y(\log D)^{\otimes_{\mathcal C} i}$ we refer to the $i$-th 
tensorial orbifold (or $\mathcal C$-)differential forms. 
See for example~\cite[Def.~2.8, Rk.~2.11]{Taji16} or Remark~\ref{tensorial}.
\end{notation}

\begin{definition}\label{multiplicity}
Let $(X,D = \sum d_i D_i)$ be snc. When $d_i\neq 1$, let $a_i$ and $b_i$ be the positive integers for 
which the equality $1-\frac{b_i}{a_i}=d_i$ holds. For every $i$, we define the $\mathcal{C}$-multiplicity of the irreducible component $D_i$ of $D$ by 
$$
  m_D(D_i) := \left\{ 
    \begin{matrix}
      \frac{1}{1-d_i}=\frac{a_i}{b_i}& \text{if $d_i\neq 1$ } \\
      \infty & \text{ if $d_i=1.$  }
    \end{matrix}
  \right.
  $$
\end{definition}


\begin{definition} Let $(X,D)$ be a smooth pair, $Y$ a smooth variety, and $\gamma:Y \to X$ a finite, flat, Galois cover with Galois group $G$ such that if $m_D(D_i)=\frac{a_i}{b_i}<\infty$, then every prime divisor in $Y$ that appears in $\gamma^*(D_i)$ has multiplicity exactly equal to $a_i$. We call $\gamma$ an adapted cover for the pair $(X,D)$, if it additionally satisfies the following properties:  
\begin{enumerate}
\item The branch locus is given by $$\Supp(H+\bigcup_{m_D(D_i)\neq \infty} D_i),$$ where H is a general member of a linear system $|L|$ of a very ample divisor $L$ in $X$.
\item $\gamma$ is totally branched over $H$.
\item $\gamma$ is not branched at the general point of $\Supp(\lfloor D \rfloor)$.
\end{enumerate}
\end{definition}

\begin{notation} Let $\gamma:Y\to X$ be an adapted cover of a smooth pair $(X,D)$, 
where $D=\sum d_iD_i$, $d_i=1-\frac{b_i}{a_i}$ as in Definition~\ref{multiplicity}. For every prime component $D_i$ of $D$ with $m_D(D_i)\neq \infty$, let $\{D_{ij}\}_{j(i)}$ be the collection of prime divisors that appear in $\gamma^{-1}(D_i)$. We define new divisors in Y by 
\begin{flalign}\label{not}
&D_Y^{i,j}:=b_iD_{ij} \ , \quad m_D(D_i)\neq \infty,\\
&D_{\gamma}:=\gamma^*(\lfloor D\rfloor).
\end{flalign}
\end{notation}

\begin{definition} Given an snc pair $(X,D)$ with an adapted 
cover $\gamma:Y\to X$, define the $\mathcal{C}$-cotangent sheaf (or orbifold cotangent sheaf) 
$\Omega^1_{(Y, \gamma, D)}$ to 
be the unique maximal locally-free subsheaf of $\Omega^1_Y(\log D_{\gamma})$ for which the sequence 
$$
\xymatrix{
0 \ar[r] & \Omega^1_{(Y, \gamma, D)} \ar[r]  
& \gamma^*\bigl(\Omega^1_X(\log (\ulcorner D \urcorner))\bigr)  \ar[r]^(.6){\rho}  
& \bigoplus \limits_{i,j(i)}\sO_{D_Y^{i,j}} \ar[r] & 0,
}
$$
induced by the natural residue map, is exact. 
\end{definition}

\begin{definition}\label{sodf} 
Let $(X,D)$ be snc, $D=\sum d_iD_i$, and $V_x$ an open neighbourhood of a given point $x\in X$ equipped with a coordinate system $z_1,\ldots,z_n$ such that $\Supp(D)\cap V_x=\{z_1\cdot \ldots\cdot z_l=0\}$, for a positive integer $1\leq l\leq n$. For every $N\in\mathbb{N}^+$, define the sheaf 
of symmetric orbifold or $\mathcal{C}$-differential forms $\Sym_{\mathcal{C}}^N\bigl(\Omega^1_X(\log D)\bigr)$ by the locally-free subsheaf of $\Sym^N \bigl(\Omega^1_X(\log(\ulcorner D \urcorner))\bigr)$ that is locally-generated, as an $\sO_{V_x}$-module, by the elements 
 $$\frac{dz_1^{k_1}}{z_1^{\lfloor d_1\cdot k_1\rfloor}} \cdot \ldots \cdot \frac{dz_l^{k_l}}{z_l^{\lfloor d_l\cdot k_l \rfloor}} \cdot dz_{l+1}^{k_{l+1}} \cdot \ldots \cdot dz_n^{k_n},$$
\noindent where $\sum k_i =N$.
\end{definition}

\begin{remark}\label{equiv}
There is an alternative definition for the sheaf of $\mathcal{C}$-differential forms: 
Let $V_x$ be an open neighbourhood of $x\in X$ as in Definition~\ref{sodf} and take $\gamma:W \to V_x$ to be an adapted cover for $(V_x,D|_{V_x})$. Let $\sigma \in \Gamma\bigl(V_x,\Sym^{N}\bigl(\Omega^1_X(*\ulcorner D \urcorner)\bigr)\bigr)$, that is $\sigma$ is a 
local rational section of $\Sym^{N}(\Omega_X)$ with poles along $\ulcorner D \urcorner$. Then,
\begin{equation}\label{alternative}
   \sigma \in \Gamma\bigl(V_x,\Sym_{\mathcal{C}}^N\bigl(\Omega^1_X (\log D)\bigr)\bigr) \iff  \gamma^{*}(\sigma)\in    \Gamma\bigl(W,\Sym^{N}
   (\Omega^1_{(W, \gamma, D|_{V_x})} \bigr),
        \end{equation}
 \noindent So that, in particular, $\gamma^*(\sigma)$ has at worst logarithmic poles \emph{only} along those prime divisors in $W$ that dominate $(\lfloor D \rfloor\cap V_x)$, and is regular otherwise.
 
 \begin{explanation} Assume that $\sigma\in \Gamma\bigl(V_x,\Sym_{\mathcal{C}}^N\bigl(\Omega^1_X(\log D)\bigr)\bigr)$ is a local $\mathcal{C}$-differential form in the sense of (\ref{alternative}). 
 It follows that $\sigma \in \Gamma\bigl(V_x,\Sym^N\bigl(\Omega^1_X\log(\ulcorner D \urcorner)\bigr)\bigr)$. In particular we find that along the reduced component of $D$ the equivalence between the two definitions trivially holds. So assume, without loss of generality, that $m_D(D_i)\neq \infty$, for all irreducible components $D_i$ of $D$. Furthermore let us assume, for simplicity, that 
$$\sigma=f\cdot \frac{dz_1^{k_1}}{z_1^{e_1}}\cdot \ldots \cdot \frac{dz_l^{k_l}}{z_l^{e_l}}\cdot dz_{l+1}^{k_l+1}\cdot \ldots \cdot dz_n^{k_n} \ \ \in \Gamma\bigl((V_x,\Sym^{N}\bigl(\Omega^1_X (\log(\ulcorner D \urcorner))\bigr)\bigr) ,$$ where $f\in \sO_{V(x)}$ with no zeros along $D_i$'s, is the local explicit description of $\sigma$. Since $\gamma^*(\sigma)\in \Sym^N(\Omega^1_{(W, \gamma, D|_{V_x})})$, the inequality 
$$k_i\cdot(a_i-1)-a_i\cdot e_i \geq k_i(b_i-1)$$ holds for $1\leq i\leq l$, where $d_i=1-(b_i/a_i)$, i.e.

$$e_i \leq k_i.d_i , \text{\ \ for all}\  1\leq i\leq l .$$ In particular $\sigma$ is a symmetric $\mathcal{C}$-differential form on $V_x$ in the sense of Definition~\ref{sodf}.
 \end{explanation}
 \end{remark}

\begin{remark}[Tensorial $\mathcal{C}$-Differential Forms]\label{tensorial} 
Similar to the Definitions~\ref{sodf} and (\ref{alternative}), we can define the sheaf of tensorial $\mathcal{C}$-differential forms $\bigl(\Omega^1_X(\log D)\bigr)^{\otimes_{\mathcal{C}}N}$ as the maximal subsheaf of $\bigl(\Omega^1_X(\log(\ulcorner D \urcorner))\bigr)^{\otimes N}$ such that 
\begin{equation*}
\gamma^*\Bigl(\bigl(\Omega^1_X(\log D)\bigr)^{\otimes_{\mathcal{C}}N}\Bigr)\subseteq (\Omega^1_{(Y, \gamma, D)})^{\otimes N}.
\end{equation*}
Using the notations in Remark~\ref{equiv}, pluri-$\mathcal{C}$-differential forms are locally defined as follows: 
\begin{equation}\label{alternative-tensor}
   \sigma \in \Gamma\bigl(V_x,\bigl(\Omega^1_X (\log D)\bigr)^{\otimes_{\mathcal{C}}N}\bigr) \iff  \gamma^{*}(\sigma)\in    
   \Gamma\bigl(W, \Omega^1_{(W, \gamma, D|_{V_x})}  \bigr) .
        \end{equation}
\end{remark}

\begin{definition}\label{ckappa}
Let $(X,D)$ be an snc pair and $\sL\subseteq \bigl(\Omega^1_X(\log D)\bigr)^{\otimes_{\mathcal{C}}r}$ a saturated coherent subsheaf of rank one. Define the $\mathcal{C}$-product $\sL^{\otimes_{\mathcal{C}}m}$ of $\sL$, to the order of $m$, to be the saturation of the image of  $\sL^{\otimes m}$ inside $\bigl(\Omega^1_X(\log D)\bigr)^{\otimes_{\mathcal{C}}(m.r)}$ and define the $\mathcal{C}$-Kodaira dimension of $\sL$ by $$\kappa_{\mathcal{C}}(X,\sL):=\max \{k \; | \; \limsup_{m \to \infty} \frac{h^0\bigl(X,\sL^{\otimes_{\mathcal{C}}m}\bigr)}{m^k} \neq 0 \}, $$ and when $h^0\bigl(X,\sL^{\otimes_{\mathcal{C}}m}\bigr)=0$ for all $m\in \mathbb{N^+}$, then, by convention, we define $\kappa_{\mathcal{C}}(X,\sL)=-\infty$.
\end{definition}

\begin{definition}[Orbifold-base; the reduced case]
\label{orbibase}
Let $(Y,D)$ be a pair and assume that $D$ is reduced and $Y$ is normal. 
Given a morphism $h: Y\to X$ with connected, positive-dimensional fibers and reduced, factorial $X$, 
let $\disc(h) = \bigcup \Delta_i$ be the union of divisorial components of $\disc(h)$. 
For every $i$, we have 
$$
h^*\Delta_i = \sum a_j \Delta'_{ij(i)} + E_i  ,
$$
where each $\Delta'_{ij}$ is prime, $a_i \in \bN$ and $E_i$ is $h$-exceptional. 
For each $i$, if $\Delta'_{ij}\subseteq \Supp(D)$, for every $\Delta'_{ij}$  in $h^{-1}(\Delta_i)$, 
set $m_{h,D}(\Delta_i) := \infty$. 
Otherwise, let 
$$
m_{h, D}(\Delta_i) :=  \min_{j} \{  a_j \;  \big|  \; \Delta'_{ij} \not\subseteq \Supp(D)    \} .
$$
Define the orbifold base (or $\mathcal C$-base) of $h:(Y,D)\to X$ by 
$$
\Delta_X(h, D) : =  \sum_i \big(  1- \frac{1}{m_{h,D}(\Delta_i)}    \big) \Delta_i   .
$$
\end{definition}

\medskip




\bibliography{bibliography/general}


\end{document}

\medskip